\documentclass{amsart}

\usepackage[usenames,dvipsnames]{color}
\usepackage{amsthm}
\usepackage{amsxtra}
\usepackage{amssymb}
\usepackage{graphicx}
\usepackage{url}
\usepackage{mathrsfs}
\usepackage{amsfonts}
\usepackage{textcomp}
\usepackage{perpage,enumerate}
\usepackage[colorlinks=true,linkcolor=Red,citecolor=Green]{hyperref}

\DeclareMathAlphabet{\mathpzc}{OT1}{pzc}{m}{it}

\DeclareMathOperator*{\supp}{supp}
\DeclareMathOperator{\WF}{WF}
\def\WFh{\WF_h}

\DeclareMathOperator*{\loc}{loc}
\DeclareMathOperator*{\comp}{comp}
\DeclareMathOperator{\Ell}{ell}

\DeclareMathOperator*{\tr}{tr}

\usepackage{chngcntr}

\graphicspath{{Figures/}}

\newtheorem{theorem}{Theorem}

\numberwithin{prop}{section}

\numberwithin{corol}{section}

\newtheorem{lemma}{Lemma}
\numberwithin{lemma}{section}

\numberwithin{conjecture}{section}

{\theoremstyle{definition}

\numberwithin{defin}{section}
}
\numberwithin{equation}{section}

\renewcommand{\Re}{\mathop{\rm Re}\nolimits}
\renewcommand{\Im}{\mathop{\rm Im}\nolimits}

\newcommand{\bl}{\begin{flushleft}}
\newcommand{\el}{\end{flushleft}}
\newcommand{\br}{\begin{flushright}}
\newcommand{\ert}{\end{flushright}}
\newcommand{\bc}{\begin{center}}
\newcommand{\ec}{\end{center}}

\newcommand{\numList}{\begin{enumerate}}
\newcommand{\enumList}{\end{enumerate}}

\newcommand{\e}{\varepsilon}

\newcommand{\re}{\mathbb{R}}

\newcommand{\la}{\langle}
\newcommand{\ra}{\rangle}

\DeclareMathOperator{\Op}{Op}
\DeclareMathOperator{\HS}{HS}
\DeclareMathOperator{\Vol}{Vol}

\newcommand{\mc}[1]{\mathcal{#1}}

\theoremstyle{remark}
\newtheorem{remark}{Remark}

\newcommand{\Cc}{C_c^\infty}
\renewcommand{\O}[1]{\mathcal{O}_{#1}}

\newcommand{\Ph}[2]{\Psi^{#1}_{h#2}}

\title[Fractal Weyl laws and wave decay for general trapping]%
{Fractal Weyl laws and wave decay\\ for general trapping}
\author{Semyon Dyatlov}
\email{dyatlov@math.mit.edu}
\address{Department of Mathematics, Massachusetts Institute of Technology,
Cambridge, MA, USA}
\author{Jeffrey Galkowski}
\email{jeffrey.galkowski@mcgill.ca}
\address{Department of Mathematics, McGill University, Montr\'eal, QC, Canada}

\begin{document}

\begin{abstract}
We prove a Weyl upper bound on the number of scattering resonances in strips for manifolds with Euclidean infinite ends.
In contrast with previous results, we do not make any strong structural assumptions
on the geodesic flow on the trapped set (such as hyperbolicity) and instead use propagation statements
up to the Ehrenfest time. By a similar method we prove a decay statement with high probability
for linear waves with random initial data. The latter statement is related
heuristically to the Weyl upper bound.
For geodesic flows with positive escape rate, we obtain a power improvement over the trivial Weyl bound
and exponential decay up to twice the Ehrenfest time.
\end{abstract}

\maketitle

\section{Introduction}

In this paper, we study asymptotics of scattering resonances and linear waves on
a $d$-dimensional noncompact Riemannian manifold $(M,g)$ with Euclidean infinite ends (see~\S\ref{s:Euclid}).
Resonances are the spectral data for the Laplacian on non-compact manifolds analogous to eigenvalues
in the compact setting. They are defined as poles of the meromorphic continuation of the $L^2$ resolvent
(see~\S\ref{s:resolvent-infinity})
\begin{equation}
  \label{e:resolvent}
R_g(\lambda)=(-\Delta_g-\lambda^2)^{-1}:\begin{cases}
L^2(M)\to L^2(M),&\Im\lambda>0,\\
L^2_{\comp}(M)\to L^2_{\loc}(M),&\lambda\in \mathbb C\setminus (-\infty,0].
\end{cases}
\end{equation}
Our results involve the structure of the homogeneous geodesic flow
\begin{equation}
  \label{e:geoflow}
\varphi_t=\exp(tH_p):T^*M\setminus 0\to T^*M\setminus 0,\qquad
p(x,\xi)=|\xi|_{g(x)}.
\end{equation}

\subsection{Weyl bounds}

Our first result is an upper bound on the number of resonances in strips,
\begin{equation}
  \label{e:NRB}
\mc{N}(R,\beta):=\#\{\lambda \in [R,R+1]+i[-\beta,0]\colon \lambda \text{ is a resonance}\},\qquad \beta\geq0,\quad
R\to\infty.
\end{equation}
We first state the following simple corollary of the main result:
\begin{theorem}
  \label{t:DG}
For all $\beta>0$ we have
\begin{equation}
  \label{e:DG-1}
\mathcal N(R,\beta)=\mathcal O(R^{d-1}).
\end{equation} 
Moreover, if the trapped set $K\subset T^*M\setminus 0$ of $\varphi_t$ has volume zero
(see~\eqref{e:gamma-def}), then
\begin{equation}
  \label{e:DG-2}
\mathcal N(R,\beta)=o(R^{d-1})\quad\text{as }R\to \infty.
\end{equation}
\end{theorem}
\noindent The bound~\eqref{e:DG-1} has previously been established in various settings by
Petkov--Zworski~\cite[(1.6)]{PetkovZworski},
Bony~\cite{Bony}, and
Sj\"ostrand--Zworski~\cite[Theorem~2]{SjZwFractal}.
We remark that in
general it is difficult to obtain lower bounds on the number of resonances in strips.
 
To state a more precise bound, we use Liouville volume of the set
of trajectories trapped for time $t$
\begin{equation}
  \label{e:V-t}
\mathcal V(t)=\mu_L(S^*M\cap\mathcal T(t)),\qquad
\mathcal T(t)=\pi^{-1}(\mathcal B)\cap \varphi_{-t}(\pi^{-1}(\mathcal B)),
\end{equation}
where $\pi:T^*M\setminus 0\to M$ is the projection map,
$S^*M=\{|\xi|_g=1\}$ is the cosphere bundle, and
$\mathcal B$ is a large compact set with smooth boundary, see~\eqref{e:the-set-B}.
We also use the \emph{Ehrenfest time} at frequency $R>0$,
\begin{equation}
  \label{e:t-e-R}
t_e(R)={\log R\over 2\Lambda_{\max}},\qquad
\Lambda_{\max}:=\limsup_{|t|\to \infty}\frac{1}{|t|}\log\sup_{(x,\xi)\in \mc{T}(t)}\|d\varphi_t(x,\xi)\|.
\end{equation}
Here $\Lambda_{\max}\in [0,\infty)$ is the maximal expansion rate and
if $\Lambda_{\max}=0$, we may replace $\Lambda_{\max}$
by an arbitrarily small positive number and accordingly take $t_e(R)=C\log R$ for any fixed constant $C$.

The following is our main Weyl bound, which immediately implies Theorem~\ref{t:DG} since
$\mathcal V(t)$ is always bounded and $\lim_{t\to\infty}\mathcal V(t)=0$ when $K$ has volume zero.
A connection between the function $\mathcal V(t)$ and resonance counting has previously been
used heuristically in the literature, see~\cite[(10)]{ZworskiRPG}.
See also Stefanov~\cite{Stefanov} for volume-based bounds on the number of resonances
polynomially close to the real axis.%
\begin{figure}
\includegraphics{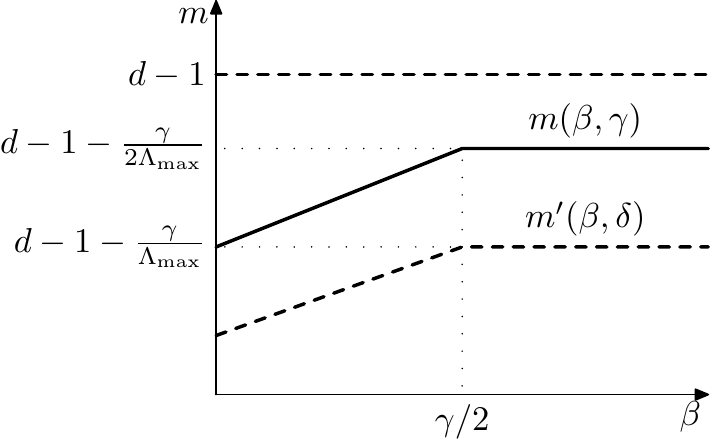}
\qquad\qquad
\includegraphics{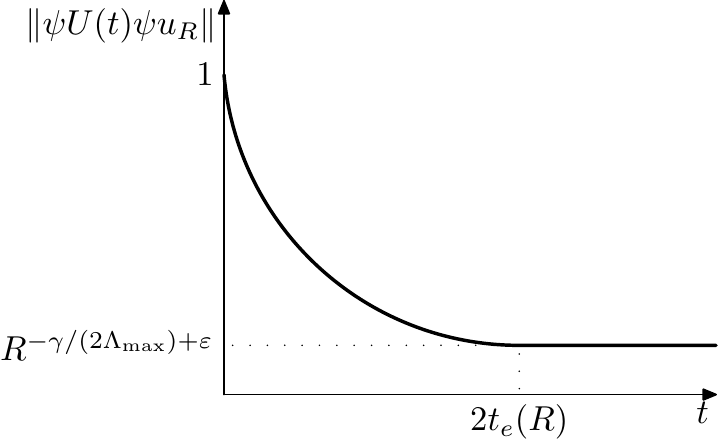}
\hbox to\hsize{\hss \qquad\qquad(a)\hss\hss\qquad\qquad (b)\hss}
\caption{(a) A plot of the exponent $m(\beta,\gamma)$ from~\eqref{e:better-bound} in the case
of positive classical expansion rate $\gamma$,
as compared to the standard Weyl law $m=d-1$ and to the exponent $m'(\beta,\delta)$
from~\cite{DyBo} in the case of hyperbolic manifolds.
(b) A plot of the typical behavior of the norm $\|\psi U(t)\psi u_R\|_{L^2}$ from Theorem~\ref{thm:decayOnAverage}.}
\label{f:weyl-plot}
\end{figure}
\begin{theorem}
\label{thm:fractalWeyl}
For each $\beta\geq 0$, $\e>0$, there exists a constant $C>0$ such that\
\begin{equation}
  \label{e:fractalWeyl}
  \mc{N}(R,\beta)\leq CR^{d-1}\min\Big[\mathcal V\big((1-\varepsilon)t_e(R)\big),\exp\big(2\beta t_e(R)\big)\cdot 
\mathcal V\big(2(1-\varepsilon) t_e(R)\big)\Big].
\end{equation}
\end{theorem}
The proof of Theorem~\ref{thm:fractalWeyl} follows the strategy of~\cite{DyBo}.
We first construct an approximate inverse for the complex scaled version of the operator
$-\Delta_g-\lambda^2$ which shows that if $\lambda$ is a resonance,
then $I-A(\lambda)$ is not invertible, where $A(\lambda)$ is a pseudodifferential
operator whose symbol is supported in a small neighborhood of the trapped set.
By Jensen's inequality, the number of resonances can be estimated using
bounds on the determinant of $I-A(\lambda)^2$, which is controlled by the Hilbert--Schmidt norm
$\|A(\lambda)\|_{\HS}$. The latter norm can be bounded by the right-hand side of~\eqref{e:fractalWeyl}.
The operator $A(\lambda)$ is defined using the dynamics of the flow for time $t_e(R)$,
and due to Egorov's theorem up to Ehrenfest time it lies in a mildly exotic pseudodifferential calculus.

The proof of Theorem~\ref{thm:fractalWeyl} only relies on propagation of singularities
and the semiclassical outgoing property of the resolvent, see~\S\ref{s:resolvent-infinity}.
In particular it applies to a wide variety of situations including semiclassical
Schr\"odinger operators and asymptotically hyperbolic manifolds (where~\cite{Vasy-AH1,Vasy-AH2} replaces
complex scaling).
It also applies to the setting of Pollicott--Ruelle resonances where upper bounds based on volume
estimation have been proved by Faure--Sj\"ostrand~\cite{Faure-Sjostrand},
Datchev--Dyatlov--Zworski~\cite{ddz},
and Faure--Tsujii~\cite{Faure-Tsujii-FWL}.

The expression~\eqref{e:fractalWeyl} can be bounded in terms of the \emph{classical escape rate}
\begin{equation}
  \label{e:escape-rate}
\gamma:=-\limsup_{t\to\infty} {1\over t}\log \mathcal V(t)\geq 0.
\end{equation}
Theorem~\ref{thm:fractalWeyl} implies that (see Figure~\ref{f:weyl-plot}(a))
\begin{equation}
  \label{e:fractalWeyl-gamma}
\mathcal N(R,\beta)=\mathcal O(R^{m(\beta,\gamma)+}),\quad
m(\beta,\gamma):=\begin{cases}
d-1-\displaystyle{\gamma-\beta\over \Lambda_{\max}},&\displaystyle 0\leq \beta\leq{\gamma\over 2}; \\
\noalign{\vskip .1in}
d-1-\displaystyle{\gamma\over 2\Lambda_{\max}},&\displaystyle \beta\geq{\gamma\over 2}.
\end{cases}
\end{equation}
where $\mathcal O(R^{m+})$ stands for a function which is $\mathcal O(R^{m+\varepsilon})$
for each $\varepsilon>0$. Note that the change in behavior for $m(\beta,\gamma)$
happens when $\beta$ is equal to half the classical escape rate, which is the depth
at which accumulation of resonances has previously
been observed mathematically, numerically, and experimentally~-- see~\S\ref{s:history}.

Under the assumption that the trapped set is \emph{hyperbolic},
there exist several previous results giving bounds on $\mathcal N(R,\beta)$
which are stronger than~\eqref{e:fractalWeyl-gamma}, see~\S\ref{s:history}.
For instance, in the case of $d$-dimensional convex co-compact hyperbolic quotients
with limit set of dimension $\delta\in [0,d-1)$ we have~\cite[Theorem~1]{DyBo}
\begin{equation}
  \label{e:better-bound}
\mathcal N(R,\beta)=\mathcal O(R^{m'(\beta,\delta)+}),\quad
m'(\beta,\delta)=\min(2\delta+2\beta+1-d,\delta).
\end{equation}
Since in this case $\gamma=d-1-\delta$ and $\Lambda_{\max}=1$,
the bound~\eqref{e:better-bound} corresponds to~\eqref{e:fractalWeyl-gamma}
with $\Lambda_{\max}$ replaced by ${1\over 2}\Lambda_{\max}$, or equivalently
$t_e(R)$ replaced by $2t_e(R)$. The lack of optimality of~\eqref{e:fractalWeyl}
is thus due to the fact that without the hyperbolicity assumption we can only propagate quantum observables up to the Ehrenfest time
(rather than twice the Ehrenfest time as in~\cite{DyBo}).
Upper bounds on $\mathcal N(R,\beta)$ are also available in the case of normally hyperbolic trapping~--
see~\S\ref{s:history}.

On the other hand, little is known on resonance bounds in strips 
for smooth metrics when $\varphi_t$
is not hyperbolic or normally hyperbolic on the trapped set, and Theorem~\ref{thm:fractalWeyl}
appears to give the first general upper bound depending on the dynamics of $\varphi_t$.
(For operators with real analytic coefficients, a bound depending
on the volume of an $R^{-1/2}$ sized neighborhood of the trapped set
was proved by Sj\"ostrand~\cite[Theorem~4.2]{SjostrandFWL}.)
In particular, if the escape rate is positive
then Theorem~\ref{thm:fractalWeyl}
gives a power improvement over $\mathcal O(R^{d-1})$.
The most promising potential example of such systems which are not hyperbolic/normally hyperbolic
is given by uniformly partially hyperbolic systems,
see~\cite{FernandoEnrique} and~\cite[Theorem 4]{YoungDyn}.

An example with zero escape rate is given by
manifolds of revolution with cylindrical or degenerate
hyperbolic trapping, where Theorem~\ref{thm:fractalWeyl}
gives an improvement which is a power of $\log R$~-- see~\S\ref{s:examples}.
See the work of Christianson~\cite{Hans} for a related question
of resolvent bounds on more general manifolds
of revolution.

\subsection{Wave decay for random initial data}

Our next theorem concerns high probability decay estimates for the half-wave group
$$
U(t):=\exp(-it\sqrt{-\Delta_g}).
$$
It is often not possible to show deterministic exponential decay for the 
cutoff propagator $\psi U(t)\psi$, $\psi\in C_c^\infty(M)$, when the trapping
is sufficiently strong. However as Theorem~\ref{thm:decayOnAverage} below shows,
if the classical escape rate is positive then such exponential decay holds
for a certain time when the initial data is random.
We apply $U(t)$ to a function chosen at random using the following procedure.
Let $\mathcal B$ be the large
smooth compact subset of $M$ given by~\eqref{e:the-set-B},
$\Delta_{\mathcal B}$ be the Dirichlet Laplacian on $\mathcal B$ with respect to the metric $g$, and
$\{(e_k,\lambda_k)\}_{k=1}^\infty$ be an orthonormal basis of $L^2(\mathcal B)$ with
$$
(-\Delta_{\mathcal B}-\lambda_k^2)e_k=0.
$$
Fix small $\varepsilon'>0$. For $R>0$ consider the subspace of $L^2(\mathcal B)$
\begin{equation}
  \label{e:E-R}
\mathcal E_R:=\bigg\{\sum_{k\in I_R}a_ke_k(x),\ a_k\in \mathbb{C}\bigg\},\quad\quad I_R:=\{k\colon \lambda_k\in R[1-\varepsilon',1+\varepsilon']\}.
\end{equation}
By the Weyl law~\cite[Theorem~29.3.3]{HOV4}, $\mathcal E_R$ has dimension
$cR^{d}+\mathcal O(R^{d-1})$ for some $c>0$. Let
$$
u_R\ \in\ \mathcal S_R:=\{u\in \mathcal E_R\colon \|u\|_{L^2}=1\}
$$
be chosen at random with respect to the standard measure on the sphere.
As before, denote by $K\subset T^*M\setminus 0$ the trapped set.
Then our result is as follows:
\begin{theorem}
  \label{thm:decayOnAverage}
Suppose that $K\neq\emptyset$ and $\psi \in C_c^\infty(\mathcal B^\circ)$.
Fix $C_0,\alpha,\varepsilon>0$.
Then there exists $C>0$ such that for all $m\geq C$, 
\begin{equation}
\label{e:decayonAverage}
\mathbb{P}\Big[\|\psi U(t)\psi u_R\|_{L^2}\leq m \sqrt{\mathcal V\big((1-\varepsilon)\min(t,2t_e(R))\big)}\text{ for all }t\in [\alpha\log R,C_0R]\Big]\ \geq\ 1-Ce^{-m^2/C}.
\end{equation}
\end{theorem}
\noindent A related result in the setting of the damped wave equation was proved
by Burq--Lebeau~\cite[page~6]{BurqLebeau}. To the authors' knowledge,
Theorem~\ref{thm:decayOnAverage} has not been previously known even
in simple settings such as a single hyperbolic trapped orbit.
We expect that a corresponding lower bound can be proved by a similar argument.

In terms of the escape rate $\gamma$ from~\eqref{e:escape-rate}, Theorem~\ref{thm:decayOnAverage}
gives the following bound with high probability for each $\varepsilon>0$ (see Figure~\ref{f:weyl-plot}(b)):
\begin{equation}
  \label{e:decayOnAverage-gamma}
\|\psi U(t)\psi u_R\|_{L^2}=\begin{cases}
\mathcal O(e^{-\gamma t/2+\varepsilon t}),& \alpha\log R\leq t\leq 2t_e(R);\\
\mathcal O(R^{-\gamma/(2\Lambda_{\max})+\varepsilon}),& 2t_e(R)\leq t\leq C_0R.
\end{cases}
\end{equation}
The bounds~\eqref{e:fractalWeyl-gamma} and~\eqref{e:decayOnAverage-gamma}
(and more generally
Theorems~\ref{thm:fractalWeyl} and~\ref{thm:decayOnAverage}) are related by the following heuristic.
To simplify the formulas below assume that $\Lambda_{\max}=1$.
Take small $\beta>0$, then by~\eqref{e:fractalWeyl-gamma} the number of resonances
in
$$
\Omega=\{\lambda\colon R/2\leq |\Re\lambda|\leq R,\quad \Im\lambda\geq -\beta\}
$$
is $\mathcal O(R^{d-\gamma+\beta+})$. Suppose that $U(t)$ has a resonance expansion up to $\Im\lambda=-\beta$
(similar to~\cite[Theorem~3.9]{ZwScat} but with infinitely many terms in the expansion;
such resonance expansions are quite rare which is one of the reasons why the argument below
is heuristic).
Then we expect for some~$N$,
\begin{equation}
  \label{e:resexp}
\psi U(t)\psi u_R=\sum_{\lambda\in\Omega\atop \lambda\text{ resonance}} e^{-it\lambda} \langle \psi u_R,v_\lambda\rangle
\psi w_\lambda+\mathcal O(R^Ne^{-\beta t})
+\mathcal O(R^{-\infty}).
\end{equation}
Here resonances with $\Im\lambda\geq -\beta$ and $|\Re\lambda|\notin [R/2,2R]$ would contribute
$\mathcal O(R^{-\infty})$ because the corresponding coresonant states live in a different band
of frequencies than $\psi u_R$.

If we additionally knew that the resonant and coresonant states $w_\lambda,v_\lambda$ are bounded in $L^2_{\loc}$
and form approximately orthonormal systems on $\supp\psi$, then with high probability
we would have $\langle \psi u_R,v_\lambda\rangle\sim R^{-d/2}$.
Estimating the norm of the sum on the right-hand side of~\eqref{e:resexp}
and using approximate orthogonality, we then expect that
$$
\|\psi U(t)\psi u_R\|_{L^2}\leq \mathcal O(R^{{\beta-\gamma\over 2}+})
+\mathcal O(R^Ne^{-\beta t}).
$$
For $t\geq C_1\log R$ and $C_1$ large enough, the first term on the right-hand side dominates
and we recover~\eqref{e:decayonAverage} (given that $\beta$ can be chosen small).
Note that~\eqref{e:decayonAverage}
also holds for $t\leq C_1\log R$, but this cannot be seen from the resonance expansion because
the error term in this expansion dominates for short times.

We remark that while the above heuristic is useful to relate Theorems~\ref{thm:fractalWeyl}
and~\ref{thm:decayOnAverage}, the proof of Theorem~\ref{thm:decayOnAverage} does not rely on it.
Instead, by a concentration of measure argument we reduce to estimating the Hilbert--Schmidt
norm of the cutoff propagator $\psi U(t)\psi$ restricted to a range of frequencies.
The latter norm is next bounded in terms of the volume $\mathcal V(t)$.
As in the proof of Theorem~\ref{thm:fractalWeyl}, this strategy can only be used up to
time $2t_e(R)$ so that the resulting symbols still lie in a mildly exotic calculus.

\subsection{Previous results}
  \label{s:history}
  
We now briefly review previous results on Weyl bounds for resonances in strips, referring the reader
to the reviews of Nonnenmacher~\cite[\S\S4,7]{NoReview} and Zworski~\cite[\S3.4]{ZwoReview}
for more information.
 
When the trapping is hyperbolic, upper bounds on $\mathcal N(R,\beta)$ have been proved in various settings
by Sj\"ostrand~\cite{SjostrandFWL}, Zworski~\cite{Zworski99},
Guillop\'e--Lin--Zworski~\cite{GLZ}, Sj\"ostrand--Zworski~\cite{SjZwFractal},
Datchev--Dyatlov~\cite{DaDy}, and
Nonnenmacher--Sj\"ostrand--Zworski~\cite{NoSjZw}. These bounds take the form
\begin{equation}
  \label{e:standardFWL}
\mathcal N(R,\beta)=\mathcal O(R^{\delta+})
\end{equation}
where $2\delta+1$ is the upper Minkowski dimension of $K\cap S^*M$,
and $R^{\delta+}$ can be replaced by $R^\delta$ if $K\cap S^*M$ has pure Minkowski dimension.
The bound~\eqref{e:standardFWL} is
stronger than the one in Theorem~\ref{thm:fractalWeyl}. Indeed,
$\varphi_{-t/2}(\mathcal T(t))$ contains an $e^{-(\Lambda_{\max}+\varepsilon)t/2}$ sized neighborhood
of the trapped set $K$, which implies that (assuming that the upper and lower Minkowski dimensions
of $K$ agree)
$$
\mathcal V((1-\varepsilon)t)\geq C^{-1}e^{-\Lambda_{\max}(d-1-\delta)t}.
$$
Therefore
$$
R^{d-1}\min\Big[\mathcal V\big((1-\varepsilon)t_e(R)\big),\exp\big(2\beta t_e(R)\big)\cdot 
\mathcal V\big(2(1-\varepsilon) t_e(R)\big)\Big]\geq C^{-1}\min\big(R^{d-1+\delta\over 2},R^{\delta+\beta/\Lambda_{\max}}\big).
$$
See also the discussion following~\eqref{e:better-bound}. 

In the setting of hyperbolic quotients, Naud~\cite{Na}, Jakobson--Naud~\cite{JaNa}, and Dyatlov~\cite{DyBo}
have obtained bounds which improve over~\eqref{e:standardFWL} when $\delta<\gamma/2$;
here $\gamma>0$ is the escape rate defined in~\eqref{e:escape-rate}. See also
the work of Dyatlov--Jin~\cite{DyatlovJin} in the case of open quantum maps. Concentration
of resonances near the line $\{\Im\lambda=-\gamma/2\}$ has been observed numerically
(for the semiclassical zeta function in obstacle scattering) by Lu--Sridhar--Zworski~\cite{LSZ}
and experimentally (for microwave scattering) by Barkhofen et al.~\cite{ZworskiPRL}.

For $r$-normally hyperbolic trapped sets (such as those appearing in Kerr--de Sitter black holes),
Dyatlov~\cite{DyJAMS} obtained an upper bound of the form~\eqref{e:standardFWL}. In this setting
$K$ is smooth and $\delta$ is an integer. Under a pinching condition, it is shown in~\cite{DyJAMS,DyAIF}
that resonances in strips have a band structure and the number of resonances in the first band
with $|\lambda|\leq R$ grows like $R^{\delta+1}$.

\subsection{Structure of the paper}

\begin{itemize}
\item In~\S\ref{s:prelims} we review geometry and dynamics
of manifolds with Euclidean ends (\S\ref{s:Euclid})
and semiclassical analysis (\S\S\ref{s:semi}, \ref{s:funcal}).
\item In~\S\ref{s:reduction} we perform analysis of the
scattering resolvent and the wave propagator near the
infinite ends of $M$ to reduce to a neighborhood of the trapped set.
\item In~\S\ref{s:cutoffs} we construct dynamical cutoff functions
used in the proofs.
\item In~\S\ref{s:weyl}, we prove Theorem~\ref{thm:fractalWeyl}.
\item In~\S\ref{s:wave}, we prove Theorem~\ref{thm:decayOnAverage}.
\item In~\S\ref{s:examples}, we estimate the quantity $\mathcal V(t)$
for two examples of manifolds of revolution. 
\end{itemize}

\noindent\textbf{Acknowledgements}. The authors would like to thank Maciej Zworski, Nicolas Burq,
St\'ephane Nonnenmacher, and
Andr\'as Vasy for many useful discussions,
and the anonymous referees for many useful suggestions to improve the manuscript.
This research was conducted during the period SD served as
a Clay Research Fellow. JG was partially supported by an NSF Mathematical Science Postdoctoral Research
Fellowship DMS-1502661.

\section{Preliminaries}
\label{s:prelims}

\subsection{Manifolds with Euclidean ends}
  \label{s:Euclid}

Thoughout the paper we assume that $(M,g)$ is a noncompact complete $d$-dimensional Riemannian manifold which
has Euclidean infinite ends in the following sense:
\begin{itemize}
\item there exists a function $r\in C^\infty(M;\mathbb R)$ such that
the sets $\{r\leq c\}$ are compact for all $c$, and
\item there exists $r_0>0$ such that $\{r\geq r_0\}$ is the disjoint union
of finitely many components, each of which is isometric to $\mathbb R^d\setminus B(0,r_0)$
with the Euclidean metric, and the pullback of $r$ under the isometry
is the Euclidean norm.
\end{itemize}
The connected components of $\{r\geq r_0\}$ are called the \emph{infinite ends} of $M$.
We parametrize each of them by a \emph{Euclidean coordinate}
$y\in\mathbb R^d\setminus B(0,r_0)$ so that
$g=\sum_{j=1}^d dy_j^2$. We lift $r$ to a function
on $T^*M$ and parametrize the cotangent bundle of each infinite end
by $(y,\eta)\in T^*(\mathbb R^d\setminus B(0,r_0))$.

As in~\eqref{e:geoflow}, put $p(x,\xi):=|\xi|_{g(x)}$
and $\varphi_t:=\exp(tH_p)$. Then
on each infinite end, we have
\begin{equation}
  \label{e:operator-far}
p(y,\eta)=|\eta|,\quad
H_{p}={\langle\eta,\partial_y\rangle\over |\eta|}.
\end{equation}
Define the \emph{directly escaping sets} in $T^*\mathbb R^d$ by
\begin{equation}
  \label{e:de-sets-R}
\begin{aligned}
\mathcal E_{\pm,\mathbb R}:=\{(y,\eta)\in T^*\mathbb R^d\colon
|y|\geq r_0,\
\pm \langle y,\eta\rangle_{\mathbb R^d}\geq 0\},\\
\mathcal E_{\pm,\mathbb R}^\circ:=\{(y,\eta)\in T^*\mathbb R^d\colon
|y|> r_0,\
\pm \langle y,\eta\rangle_{\mathbb R^d}>0\},
\end{aligned}
\end{equation}
and pull these back by the Euclidean coordinates in the infinite ends of $M$ to
\begin{equation}
  \label{e:de-sets-M}
\mathcal E_\pm,\mathcal E_\pm^\circ\ \subset\ \{r\geq r_0\}\ \subset\ T^*M.
\end{equation}
It follows from~\eqref{e:operator-far} that for $\mathbf x\in T^*M\setminus 0$,
\begin{equation}
  \label{e:convexinf-1}
\mathbf x\in \mathcal E_\pm \quad \Longrightarrow\quad
\varphi_{\pm t}(\mathbf x)\in\mathcal E_\pm,\quad
r(\varphi_{\pm t}(\mathbf x))\geq\sqrt{r(\mathbf x)^2+t^2} \quad\text{for all }t\geq 0,
\end{equation}
in particular $r(\varphi_t(\mathbf x))\to\infty$
as $t\to \pm\infty$. Arguing by contradiction, this implies that for all $\mathbf x\in T^*M\setminus 0$
\begin{equation}
  \label{e:convexinf-2}
r(\mathbf x)\geq r_0,\
r(\varphi_{\mp t_0}(\mathbf x))\leq r(\mathbf x)\text{ for some }t_0> 0\quad\Longrightarrow\quad
\pm\langle y(\mathbf x),\eta(\mathbf x)\rangle_{\mathbb R^d}> 0.
\end{equation}
Therefore, if a trajectory of $\varphi_t$ starting on $\{r<r_0\}$ enters some infinite end, it escapes
to infinity inside this end.

Define the incoming/outgoing tails $\Gamma_\pm$ and the trapped set $K$ by
\begin{equation}
  \label{e:gamma-def}
\Gamma_\pm:=\{\mathbf x\in T^*M\setminus 0\ \colon\ r(\varphi_t(\mathbf x))\not\to\infty\text{ as }t\to\mp\infty\},\quad
K:=\Gamma_+\cap\Gamma_-.
\end{equation}
The next lemma establishes basic properties of $\Gamma_\pm$ and $K$;
see~\cite[\S6.1]{ZwScat} for a more general setting.
\begin{lemma}
  \label{l:trapprop}
1. The sets $\Gamma_\pm,K$ are closed in $T^*M\setminus 0$ and
\begin{equation}
  \label{e:K-contained}
K\subset \{r<r_0\},
\end{equation}
in particular $K\cap S^*M$ is compact.

\noindent 2. We have locally uniformly in $\mathbf x$,
\begin{equation}
  \label{e:trapprop-1}
\mathbf x\in \Gamma_\pm \quad\Longrightarrow\quad d(\varphi_t(\mathbf x),K)\to 0\quad\text{as }t\to \mp\infty.
\end{equation}

\noindent 3. Let $U$ be a neighborhood of $K$ and $V\subset T^*M\setminus 0$ be compact. Then
there exists $T>0$ such that
\begin{equation}
  \label{e:trapprop-2}
\varphi_{-t}(V)\cap\varphi_{s}(V)\ \subset\ U\quad\text{for all }t,s\geq T.
\end{equation}

\noindent 4. Assume that $V\subset T^*M\setminus 0$ is compact and
$V\cap\Gamma_\pm=\emptyset$. Then there exists $T>0$ such that
\begin{equation}
  \label{e:trapprop-3}
\varphi_{\mp t}(V)\ \subset\ \mathcal E_\mp^\circ\cap\Big\{r\geq \sqrt{r_0^2+(t-T)^2}\Big\}\quad\text{for all }t\geq T.
\end{equation}
Moreover, the set $\bigcup_{\mp t\geq 0}\varphi_t(V)$ is closed in $T^*M$.
\end{lemma}
\begin{proof}
1. We first show that $\Gamma_-$ is closed in $T^*M\setminus 0$. Assume
that $\mathbf x_0\in T^*M\setminus 0$ and $\mathbf x_0\notin\Gamma_-$. Then
$r(\varphi_t(\mathbf x_0))\to\infty$ as $t\to \infty$, thus by~\eqref{e:convexinf-2}
there exists $t_0>0$ such that
$\varphi_{t_0}(\mathbf x_0)\in\mathcal E_+^\circ$.
Since $\mathcal E_+^\circ$ is open,
we have $\varphi_{t_0}(\mathbf x)\in\mathcal E_+^\circ$
for all $\mathbf x$ which are sufficiently close to $\mathbf x_0$.
By~\eqref{e:convexinf-1}, we have $\mathbf x\notin \Gamma_-$, showing that
$\mathbf x_0$ does not lie in the closure of $\Gamma_-$.
A similar argument shows that $\Gamma_+$, and thus $K$, is closed.

It remains
to show~\eqref{e:K-contained}. Assume that $\mathbf x\in T^*M\setminus 0$
and $r(\mathbf x)\geq r_0$. If $\langle y(\mathbf x),\eta(\mathbf x)\rangle_{\mathbb R^d}\geq 0$,
then by~\eqref{e:convexinf-1} we have $\mathbf x\notin\Gamma_-$. Similarly
if $\langle y(\mathbf x),\eta(\mathbf x)\rangle_{\mathbb R^d}\leq 0$,
then $\mathbf x\notin\Gamma_+$.

\noindent 2. We consider the case of $\Gamma_-$; the case
of $\Gamma_+$ is handled similarly. Assume~\eqref{e:trapprop-1} is false.
Then there exists $\varepsilon>0$ and sequences
$\mathbf x_k\in\Gamma_-$, $t_k\to\infty$ such that
$\mathbf x_k$ lie in a compact subset of $T^*M\setminus 0$
and $d(\varphi_{t_k}(\mathbf x_k),K)>\varepsilon$.
By~\eqref{e:convexinf-1} and~\eqref{e:convexinf-2},
$\mathbf x_k\in\Gamma_-$ implies that
$r(\varphi_{t_k}(\mathbf x_k))$ is bounded, specifically
$$
r(\varphi_{t_k}(\mathbf x_k))\leq \max(r(\mathbf x_k),r_0)\quad\text{when }t_k\geq 0.
$$
By passing to a subsequence, we may assume that
$$
\varphi_{t_k}(\mathbf x_k)\to \mathbf x_{\infty}\in T^*M\setminus 0.
$$
We have $\mathbf x_{\infty}\notin K$; however, since $\Gamma_-$
is closed and invariant under the flow, $\mathbf x_\infty\in\Gamma_-$.
Therefore $\mathbf x_{\infty}\notin\Gamma_+$. By~\eqref{e:convexinf-2},
there exists $T>0$ such that
$\varphi_{-T}(\mathbf x_{\infty})\in \mathcal E_-^\circ$.
Then for large enough $k$,
$\varphi_{t_k-T}(\mathbf x_k)\in\mathcal E_-^\circ$.
It follows from~\eqref{e:convexinf-1} applied to $\varphi_{t_k-T}(\mathbf x_k)$ that as $k\to\infty$,
$$
r(\mathbf x_k)=r\big(\varphi_{-(t_k-T)}(\varphi_{t_k-T}(\mathbf x_k))\big)\ \geq\ \sqrt{r_0^2+(t_k-T)^2}\ \to\ \infty,
$$
contradicting the fact that $\mathbf x_k$ varies in a compact set.

\noindent 3. Assume~\eqref{e:trapprop-2} is false. Then there exist sequences
$$
t_k,s_k\to\infty,\quad
\mathbf x_k\in \varphi_{-t_k}(V)\cap\varphi_{s_k}(V),\quad
\mathbf x_k\notin U.
$$
By~\eqref{e:convexinf-1}, assuming $t_k,s_k\geq 0$, we have
$$
r(\mathbf x_k)\leq \max(\max\nolimits_V r,r_0).
$$
Passing to a subsequence, we may assume
$$
\mathbf x_k\to \mathbf x_\infty\in T^*M\setminus 0.
$$
We have $\mathbf x_\infty\notin K$, thus $\mathbf x_\infty\notin\Gamma_+$
or $\mathbf x_\infty\notin\Gamma_-$. We assume $\mathbf x_\infty\notin\Gamma_-$,
the other case being handled similarly. By~\eqref{e:convexinf-2}, there exists
$T>0$ such that
$\varphi_T(\mathbf x_\infty)\in\mathcal E_+^\circ$. Therefore,
for $k$ large enough we have
$\varphi_T(\mathbf x_k)\in\mathcal E_+^\circ$.
It follows from~\eqref{e:convexinf-1} applied to $\varphi_T(\mathbf x_k)$ that as $k\to\infty$,
$$
r(\varphi_{t_k}(\mathbf x_k))=r\big(\varphi_{t_k-T}(\varphi_T(\mathbf x_k))\big)
\ \geq\ \sqrt{r_0^2+(t_k-T)^2}\ \to \infty
$$
contradicting the fact that $\varphi_{t_k}(\mathbf x_k)\in V$.

\noindent 4. We assume $V\cap \Gamma_-=\emptyset$, the case
$V\cap\Gamma_+=\emptyset$ being handled similarly.
Arguing as in part~1, we see that each $\mathbf x_0\in V$
has an open neighborhood $U(\mathbf x_0)$ such that for some
$T=T(\mathbf x_0)>0$ and all $\mathbf x\in U(\mathbf x_0)$, we have
$\varphi_T(\mathbf x)\in\mathcal E_+^\circ$. By~\eqref{e:convexinf-1}
applied to $\varphi_T(\mathbf x)$,
$$
\varphi_t(\mathbf x)\in \mathcal E_+^\circ\cap \Big\{r\geq \sqrt{r_0^2+(t-T(\mathbf x_0))^2}\Big\}\quad\text{for all }
\mathbf x\in U(\mathbf x_0),\
t\geq T(\mathbf x_0).
$$
To show~\eqref{e:trapprop-3}, it remains to cover $V$ by finitely many open
sets of the form $U(\mathbf x_0)$ and let $T$ be the maximum of the corresponding times
$T(\mathbf x_0)$.

To show that $\bigcup_{t\geq 0} \varphi_t(V)$ is closed,
take sequences $\mathbf x_j\in V$, $t_j\geq 0$,
and assume that $\varphi_{t_j}(\mathbf x_j)$
converges to some $\mathbf y_\infty\in T^*M$. Then $r(\varphi_{t_j}(\mathbf x_j))$
is bounded, so by~\eqref{e:trapprop-3} the sequence $t_j$ is bounded as well.
Passing to subsequences, we may assume that $t_j\to t_\infty\geq 0$,
$\mathbf x_j\to \mathbf x_\infty\in V$. Then
$\mathbf y_\infty =\varphi_{t_\infty}(\mathbf x_\infty)\in \bigcup_{t\geq 0} \varphi_t(V)$,
finishing the proof.
\end{proof}
Following~\eqref{e:V-t} we define for $\mathcal B\subset M$
$$
\mathcal V_{\mathcal B}(t):=\mu_L(S^*M\cap \mathcal T_{\mathcal B}(t)),\quad
\mathcal T_{\mathcal B}(t):=\pi^{-1}(\mathcal B)\cap \varphi_{-t}(\pi^{-1}(\mathcal B)).
$$
By~\eqref{e:trapprop-2}, if $\pi^{-1}(\mathcal B)$ contains a neighborhood of $K$
and $\mathcal B'\subset M$ is compact, then there exists $T>0$ such that
$$
\mathcal T_{\mathcal B'}(t+2T)\ \subset\ \varphi_{-T}(\mathcal T_{\mathcal B}(t)),\quad
t\geq 0,
$$
thus in particular
\begin{equation}
  \label{e:B-does-not-matter}
\mathcal V_{\mathcal B'}(t+2T)\leq \mathcal V_{\mathcal B}(t),\quad
t\geq 0.
\end{equation}
Since Theorems~\ref{thm:fractalWeyl} and~\ref{thm:decayOnAverage}
use quantities of the form $\mathcal V((1-\varepsilon)t)$ where
$t\geq C^{-1}\log R$, by slightly changing~$\varepsilon$
and using~\eqref{e:B-does-not-matter} we see that
these theorems do not depend on the choice of $\mathcal B$,
as long as $\pi^{-1}(\mathcal B)$ contains a neighborhood of $K$.
We henceforth fix $r_1>r_0$ and put
\begin{equation}
  \label{e:the-set-B}
\mathcal B:=\{r\leq r_1\}.
\end{equation}
By~\eqref{e:convexinf-1}, the set $\mathcal B$ is geodesically convex, therefore
$$
\mathcal T_{\mathcal B}(t+t_0)\subset \varphi_{-t_0}\big(\mathcal T_{\mathcal B}(t)\big)\quad\text{for all }t,t_0\geq 0,
$$
implying that
\begin{equation}
  \label{e:volume-decreasing}
\mathcal V_{\mathcal B}(t+t_0)\leq \mathcal V_{\mathcal B}(t)\quad\text{for all }t,t_0\geq 0.  
\end{equation}
Moreover, if $K\cap S^*M\neq \emptyset$, then we have for each $\Lambda>\Lambda_{\max}$,
\begin{equation}
  \label{e:volume-nonzero}
\mathcal V_{\mathcal B}(t)\geq C^{-1}e^{-2(d-1)\Lambda t},\quad
t\geq 0.
\end{equation}
Indeed, if $(x_0,\xi_0)\in K\cap S^*M$, then $\mathcal T_{\mathcal B}(t)\cap S^*M$ contains
an $e^{-\Lambda t}$ sized neighborhood of $\varphi_s(x_0,\xi_0)$ for all
$s\in [0,1]$.

\subsection{Semiclassical analysis}
  \label{s:semi}

We next briefly review the tools from semiclassical analysis used in this paper,
referring the reader to~\cite{EZB} and~\cite[Appendix~E]{ZwScat} for a comprehensive
introduction to the subject.

For an $h$-dependent family of smooth functions $a(x,\xi;h)$ on $T^*M$,
we say that $a$ lies in the symbol class $S^m_{h,\nu}(T^*M)$ if
it satisfies the following derivative bounds on $T^*M$, uniformly in $h$:
$$
|\partial^\alpha_y\partial^\beta_\eta a(y,\eta;h)|\leq C_{\alpha\beta} h^{-\nu(|\alpha|+|\beta|)}
\langle\eta\rangle^{m-|\beta|}.
$$
Here $\nu\in [0,1/2)$ and $m\in\mathbb R$ are parameters;
$y$ is any coordinate system on $M$ which coincides with the Euclidean
coordinate in each infinite end. Note that we require the bounds
to be uniform as $y\to\infty$.

We fix a quantization procedure $\Op_h$, mapping each $a\in S^m_{h,\nu}(T^*M)$
to an $h$-dependent family of operators
$$
\Op_h(a):\ \mathscr S(M)\to \mathscr S(M),\quad
\mathscr S'(M)\to\mathscr S'(M).
$$
Here $\mathscr S(M)$ denotes the space of Schwartz functions
and $\mathscr S'(M)$ the space of tempered distributions on~$M$,
defined using Euclidean coordinates in the infinite ends.
In case $M=\mathbb R^d$, $\Op_h(a)$ is defined by the standard formula
\begin{equation}
  \label{e:standard-op}
\Op_h(a)u(x)=(2\pi h)^{-d}\int_{\mathbb R^{2d}}e^{{i\over h}\langle x-y,\xi\rangle}a(x,\xi)u(y)\,dyd\xi,
\end{equation}
and for general $M$ it is constructed from~\eqref{e:standard-op} using coordinate charts (taking
the Euclidean coordinate in each infinite end of $M$) and a partition of unity,
see for instance~\cite[Proposition~E.14]{ZwScat}. We also arrange so that
\begin{equation}
  \label{e:op-h-1}
\Op_h(1)=I.
\end{equation}
This gives a class of operators
(which is independent of the choice of coordinate charts; see below
for the definition of $h^\infty\Psi^{-\infty}(M)$)
$$
\Psi^{m}_{h,\nu}(M)=\{\Op_h(a)+\mathcal O(h^\infty)_{\Psi^{-\infty}(M)}\colon a\in S^m_{h,\nu}(T^*M)\}.
$$
The principal symbol map
$$
\sigma_h:\Psi^m_{h,\nu}(M)\to {S^m_{h,\nu}(T^*M)/ h^{1-2\nu} S^{m-1}_{h,\nu}(T^*M)},\quad
\sigma_h(\Op_h(a))=a,
$$
is independent of the choice of local coordinates and satisfies for
$A\in\Psi^m_{h,\nu}(M)$, $B\in\Psi^{m'}_{h,\nu}(M)$
\begin{align}
  \label{e:bca-1}
\sigma_h(A^*)&=\overline{\sigma_h(A)}+\mathcal O(h^{1-2\nu})_{S^{m-1}_{h,\nu}},\\
  \label{e:bca-2}
\sigma_h(AB)&=\sigma_h(A)\sigma_h(B)+\mathcal O(h^{1-2\nu})_{S^{m+m'-1}_{h,\nu}},\\
  \label{e:bca-3}
\sigma_h([A,B])&=-ih\{\sigma_h(A),\sigma_h(B)\}+\mathcal O(h^{2(1-2\nu)})_{S^{m+m'-2}_{h,\nu}}.
\end{align}
We have $\sigma_h(A)=0$ if and only if $A\in h^{1-2\nu}\Psi^{m-1}_{h,\nu}(M)$.
Every $A\in\Psi^m_{h,\nu}(M)$ is bounded uniformly in $h$ as an operator
$$
A:H^{s}_h(M)\to H^{s-m}_h(M),\quad
s\in\mathbb R,
$$
where $H^s_h(M)$ is the (global) semiclassical Sobolev space, defined using
Euclidean coordinates in the infinite ends (see~\cite[\S E.1.6]{ZwScat}).
See for instance~\cite[Theorems~4.14, 9.5, 14.1, 14.2]{EZB} for the proofs in the case $\nu=0$,
which adapt directly to the case of general $\nu$ (see~\cite[Theorems~4.17, 4.18]{EZB}).
We also have for all $A\in\Psi^0_{h,\nu}(M)$,
\begin{equation}
  \label{e:opernorm}
\|A\|_{L^2(M)\to L^2(M)}\leq \sup |\sigma_h(A)|+\mathcal O(h^{1/2-\nu}).
\end{equation}
See for instance~\cite[Theorem~5.1]{EZB} whose proof adapts
to operators in $\Psi^0_{h,\nu}$.
Using the explicit formula for the integral kernel of $\Op_h(a)$,
we also have the Hilbert--Schmidt bound
\begin{equation}
  \label{e:hsnorm}
\|\Op_h(a)\|_{\HS}^2\leq C^2h^{-d} \Vol(\supp a),\quad
a\in S^0_{h,\nu}.
\end{equation}
where $C$ is some $S^0_{h,\nu}$ seminorm of $a$.

The residual class for $S^m_{h,\nu}(M)$, denoted by $h^\infty\Psi^{-\infty}(M)$
or $\mathcal O(h^\infty)_{\Psi^{-\infty}(M)}$,
is defined as follows:
$$
A\in h^\infty\Psi^{-\infty}(M)\quad\Longleftrightarrow\quad
\|A\|_{H^{-N}_h(M)\to H^N_h(M)}\leq C_Nh^N\quad\text{for all }N.
$$
We also use the class of compactly microlocalized operators
$$
\Psi^{\comp}_{h,\nu}(M)=\{A=\Op_h(a)+\mathcal O(h^\infty)_{\Psi^{-\infty}}\mid
a\in \Cc(T^*M)\}.
$$
The standard classes of symbols and operators are given by the case $\nu=0$:
$$
S^m_h(T^*M):=S^m_{h,0}(T^*M),\quad
\Psi^m_h(M):=\Psi^m_{h,0}(M),\quad
\Psi^{\comp}_h(M):=\Psi^{\comp}_{h,0}(M).
$$
We have the following improvement of~\eqref{e:bca-3} when
$M=\mathbb R^d$, the quantization~\eqref{e:standard-op} is used,
and one of the symbols in question is in $S^m_h$:
\begin{equation}
  \label{e:bca-4}
a\in S^m_h(T^*\mathbb R^d),
b\in S^{m'}_{h,\nu}(T^*\mathbb R^d)\ \Longrightarrow\
[\Op_h(a),\Op_h(b)]=-ih\Op_h(\{a,b\})+\mathcal O(h^{2-2\nu})_{\Psi^{m+m'-2}_{h,\nu}(\mathbb R^d)}.
\end{equation}
This follows immediately from the
asymptotic expansion for the full symbol of $\Op_h(a)\Op_h(b)$, see~\cite[Theorems~4.14, 4.17]{EZB}.

For $A\in\Psi^m_{h,\nu}(M)$, the wavefront set
$\WFh(A)\subset \overline T^*M$ is defined as follows:
$(x_0,\xi_0)\in \overline T^*M$ does \emph{not} lie in $\WFh(A)$
if and only if $A=\Op_h(a)+\mathcal O(h^\infty)_{\Psi^{-\infty}}$
for some $a\in S^m_{h,\nu}(M)$ such that
$a=\mathcal O(h^\infty\langle\xi\rangle^{-\infty})$
in a neighborhood of $(x_0,\xi_0)$ in $\overline T^*M$.
Here $\overline T^*M$ is the fiber-radially compactified
cotangent bundle, see for instance~\cite[\S\S E.1.2, E.2.1]{ZwScat}.
For $A,B\in\Psi^m_{h,\nu}(M)$ and some $h$-independent open set
$U\subset \overline T^*M$, we say
$$
A=B+\mathcal O(h^\infty)_{\Psi^{-\infty}}\quad\text{microlocally in }U,
$$
if $\WFh(A-B)\cap U=\emptyset$.
For $A\in\Psi^m_{h,\nu}(M)$, the elliptic set $\Ell_h(A)\subset \overline T^*M$ is defined
as follows: $(x,\xi)\in\Ell_h(A)$ if $\langle\xi\rangle^{-m}\sigma_h(A)$
is bounded away from zero in a neighborhood of $(x,\xi)$.

\subsection{Functional calculus and the half-wave propagator}
\label{s:funcal}

By the functional calculus of self-adjoint operators in $\Psi^m_h(M)$
(see for instance~\cite[\S8]{d-s}), for each $\psi\in \Cc(\mathbb R)$ the operator
$$
\psi(-h^2\Delta_g):L^2(M)\to L^2(M)
$$
lies in $\Psi^{-N}_h(M)$ for each $N$. Moreover,
$$
\sigma_h(\psi(-h^2\Delta_g))=\psi\big(|\xi|_g^2\big),\quad
\WFh(\psi(-h^2\Delta_g))\subset \{|\xi|_g^2\in\supp\psi\},
$$
and for each open set $U\subset \mathbb R$,
\begin{equation}
  \label{e:vaelor}
\psi=1\text{ on }U\quad\Longrightarrow\quad
\psi(-h^2\Delta_g)=I+\mathcal O(h^\infty)_{\Psi^{-\infty}}\text{ microlocally in }\{|\xi|_g^2\in U\}.
\end{equation}
This makes it possible to describe
the square root $\sqrt{-\Delta_g}$ microlocally in $T^*M\setminus 0$:
\begin{lemma}
  \label{l:sqrt}
Assume that $A\in\Psi^{\comp}_h(M),\WFh(A)\subset T^*M\setminus 0$.
Then for each $N$, with $p(x,\xi)=|\xi|_{g(x)}$,
\begin{equation}
  \label{e:sqrt}
\begin{gathered}
h\sqrt{-\Delta_g}A,\ Ah\sqrt{-\Delta_g}\ \in\ \Psi^{-N}_h(M),\quad
\sigma_h(h\sqrt{-\Delta_g}A)=\sigma_h(Ah\sqrt{-\Delta_g})=p\cdot\sigma_h(A);\\
\WFh(h\sqrt{-\Delta_g}A),\WFh(Ah\sqrt{-\Delta_g})\subset\WFh(A).
\end{gathered}
\end{equation}
\end{lemma}
\begin{proof}
We consider the case of the operator $h\sqrt{-\Delta_g}A$.
Fix $C>0$ such that $\WFh(A)\subset \{C^{-1}\leq |\xi|_g^2\leq C\}$.
Choose $\psi\in\Cc((0,\infty))$ such that 
$\psi=1$ near $[C^{-1},C]$. Then by~\eqref{e:vaelor}
$$
A=\psi(-h^2\Delta_g)A+\mathcal O(h^\infty)_{\Psi^{-\infty}}.
$$
Put $\varphi(\lambda)=\sqrt{\lambda}\psi(\lambda)$, then $\varphi\in\Cc(\mathbb R)$
and
$$
h\sqrt{-\Delta_g}A=\varphi(-h^2\Delta_g)A+\mathcal O(h^\infty)_{\Psi^{-\infty}}
$$
and~\eqref{e:sqrt} follows.
\end{proof}
We next prove a Egorov theorem for the half-wave propagator
$$
U(t)=\exp(-it\sqrt{-\Delta_g}):L^2(M)\to L^2(M).
$$
Recall that $\varphi_t=\exp(tH_p)$ is the homogeneous geodesic flow on $T^*M\setminus 0$.
\begin{lemma}
  \label{l:egorov}
Assume that $a\in S^0_{h,\nu}(T^*M)$ for some $\nu\in[0,1/2)$
and $\supp a$ is contained in an $h$-independent compact subset of $T^*M\setminus 0$.
Then there exists a smooth family of symbols compactly supported in $T^*M\setminus 0$
$$
a_t\in S^0_{h,\nu}(T^*M),\quad
t\in\mathbb R;\quad
\supp a_t\subset \varphi_{-t}(\supp a),\quad
a_t=a\circ \varphi_t+\mathcal O(h^{1-2\nu})_{S^0_{h,\nu}},
$$
such that, with constants in the remainder uniform as long as $t$ is in a bounded set
$$
U(-t)\Op_h(a)U(t)=\Op_h(a_t)+\mathcal O(h^\infty)_{\Psi^{-\infty}}.
$$
\end{lemma}
\begin{proof}
Since $U(t)$ is bounded on all Sobolev spaces, it suffices to construct
$a_t$ such that
\begin{equation}
  \label{e:gorov-1}
a_0=a,\quad
d_t\big(U(t)\Op_h(a_t)U(-t)\big)=\mathcal O(h^\infty)_{\Psi^{-\infty}}.
\end{equation}
Using a partition of unity for $a$, it suffices to consider
the case when $\supp a$ is contained in a coordinate chart on $M$. Moreover,
by induction on time we see that it is enough to study the case
when $t$ is small and thus $\varphi_{-s}(a)$ lies in a fixed coordinate chart
for all $s$ between $0$ and $t$. We thus reduce to the case when $M=\mathbb R^d$
and $\Op_h$ is given by~\eqref{e:standard-op}.

The differential equation in~\eqref{e:gorov-1} can be rewritten as  
\begin{equation}  
  \label{e:gorov-2}
\Op_h(\partial_t a_t)+{i\over h}[\Op_h(a_t),h\sqrt{-\Delta_g}]=\mathcal O(h^\infty)_{\Psi^{-\infty}}.
\end{equation}
We construct $a_t$ as an asymptotic series
\begin{equation}
  \label{e:asex}
a_t\sim \sum_{j=0}^\infty a^{(j)}_t,\quad
a^{(j)}_t\in h^{j(1-2\nu)}S^0_{h,\nu}(T^*\mathbb R^d),\quad
\supp a^{(j)}_t\subset \varphi_{-t}(\supp a).
\end{equation}
To satisfy~\eqref{e:gorov-2} it suffices take $a^{(j)}_t$ such that for
some symbols
$$
b^{(j)}_t\in h^{j(1-2\nu)}S^0_{h,\nu}(T^*\mathbb R^d),\quad
\supp b^{(j)}_t\subset \varphi_{-t}(\supp a),\quad
b^{(0)}_t=0,
$$
we have
\begin{equation}  
  \label{e:gorov-3}
\Op_h(\partial_t a^{(j)}_t)+{i\over h}[\Op_h(a_t^{(j)}),h\sqrt{-\Delta_g}]+\Op_h(b_t^{(j)})=\Op_h(b_t^{(j+1)})+\mathcal O(h^\infty)_{\Psi^{-\infty}}.
\end{equation}
We construct $a^{(j)}_t,b^{(j+1)}_t$ by induction, assuming $b^{(j)}_t$ is already known.
Since $a_t^{(j)}$ is compactly supported in $T^*M\setminus 0$,
by Lemma~\ref{l:sqrt} and~\eqref{e:bca-4} the left-hand side of~\eqref{e:gorov-3} is
$$
\Op_h\big(\partial_t a^{(j)}_t-H_pa^{(j)}_t+b^{(j)}_t\big)+\mathcal O(h^{(j+1)(1-2\nu)})_{\Psi^0_{h,\nu}(\mathbb R^d)}.
$$
Then~\eqref{e:gorov-3} holds for some $b^{(j+1)}_t\in h^{(j+1)(1-2\nu)}S^0_{h,\nu}(T^*\mathbb R^d)$
if $a^{(j)}_t$ satisfies the transport equation
\begin{equation}
  \label{e:gorov-4}
\partial_t a^{(j)}_t=H_pa^{(j)}_t-b^{(j)}_t.
\end{equation}
We now put
$$
a^{(j)}_t:=\delta_{j0}(a\circ\varphi_t)-\int_0^t b^{(j)}_s\circ\varphi_{t-s}\,ds.
$$
Then~\eqref{e:gorov-4} is satisfied and thus~\eqref{e:gorov-3} holds for some
choice of $b_t^{(j+1)}$. The support condition on $a^{(j)}_t$ follows from
the support condition on $b^{(j)}_s$. The support
condition on $b_t^{(j+1)}$ follows from this and the fact that the asymptotic
expansion for the full symbol of the left-hand side of~\eqref{e:gorov-3} at each
point only depends on the values of all derivatives of
$a^{(j)}_t,b_t^{(j)}$ at this point. With $a_t$ given by~\eqref{e:asex}
we also have $a_0=a$ and $a_t=a\circ \varphi_t+\mathcal O(h^{1-2\nu})$,
finishing the proof.
\end{proof}
Lemma~\ref{l:egorov} gives us the following approximate inverse statement
for the semiclassical Helmholtz operator $-h^2\Delta_g-\omega^2$, which is a version
of propagation of singularities used in the proof of Lemma~\ref{l:iterated-parametrix}.
\begin{lemma}
  \label{l:basic-parametrix}
Assume that $a,b\in S^0_{h,\nu}(T^*M)$ are supported in an $h$-independent
compact subset of $T^*M\setminus 0$, $B'\in\Psi^{0}_h(M)$ is compactly supported,
and for some $T\geq 0$,
\begin{equation}
  \label{e:basic-parametrix-assumption}
\varphi_{-T}(\supp a)\cap \supp(1-b)=\emptyset,\qquad
\WFh(I-B')\cap \bigcup_{t=0}^T \varphi_{-t}(\supp a)=\emptyset.
\end{equation}
Then for any constant $C$ and $\omega\in [C^{-1},C]+ih[-C,C]$, we have
\begin{equation}
  \label{e:basic-parametrix}
\Op_h(a)=Z(\omega)B'(-h^2\Delta_g-\omega^2)+e^{i\omega T/h}\Op_h(a)U(T)\Op_h(b)+\mathcal O(h^\infty)_{\Psi^{-\infty}}
\end{equation}
where $Z(\omega)$ is holomorphic in $\omega$ and satisfies the estimate for all $N$,
$$
\|Z(\omega)\|_{H_h^{-N}(M)\to H_h^{N}(M)}\leq C_Nh^{-1}|\sup a|.
$$
\end{lemma}
\begin{proof}
Observe that 
$$
hD_t\big(e^{i\omega t\over h}U(t)\big)=e^{i\omega t\over h}U(t)(-h\sqrt{-\Delta_g}+\omega),\quad\quad U(0)=I.
$$
Therefore,
\begin{equation}
  \label{e:identity}
\begin{aligned}
I&=e^{i\omega T\over h}U(T)+\frac{i}{h}\int_0^T e^{\frac{i\omega t}{h}}U(t)(h\sqrt{-\Delta_g}-\omega)dt\\
&=e^{\frac{i\omega T}{h}}U(T)+\frac{i}{h}\int_0^Te^{\frac{i\omega t}{h}}U(t)(h\sqrt{-\Delta_g}+\omega)^{-1}(-h^2\Delta_g-\omega^2)dt.
\end{aligned}
\end{equation}
By~\eqref{e:op-h-1}, Lemma~\ref{l:egorov}, and~\eqref{e:basic-parametrix-assumption}, we have
$$
\begin{aligned}
\Op_h(a)U(T)(I-\Op_h(b))&=\mathcal O(h^\infty)_{\Psi^{-\infty}},\\
\Op_h(a)U(t)(h\sqrt{-\Delta_g}+\omega)^{-1}(I-B')&=\mathcal O(h^\infty)_{\Psi^{-\infty}}\quad\text{for all }t\in [0,T],
\end{aligned}
$$
where $U(-t)\Op_h(a)U(t)(h\sqrt{-\Delta_g}+\omega)^{-1}$ is a pseudodifferential
operator similarly to~\eqref{e:sqrt}.
It remains to apply $\Op_h(a)$ on the left to~\eqref{e:identity} and put
$$
Z(\omega):={i\over h}\int_0^Te^{i\omega t\over h}\Op_h(a)U(t)(h\sqrt{-\Delta_g}+\omega)^{-1}\,dt.\qedhere
$$
\end{proof}
We finally establish properties of certain spectral cutoffs of width $h$ for the operator
$h^2\Delta_g$:
\begin{lemma}
  \label{l:h-functional-calculus}
Assume that $\psi\in C^\infty(\mathbb R)$ is bounded and its Fourier transform $\widehat\psi$
satisfies for some $T_0,T_1\in\mathbb R$
\begin{equation}
  \label{e:psi-supported}
\supp\widehat\psi\subset (T_0,T_1).
\end{equation}
For $\omega\in\mathbb C$ varying in an $h$-sized neighborhood of 1, define
$B(\omega):=\psi\big({-h^2\Delta_g-\omega^2\over h}\big):L^2(M)\to L^2(M)$,
where $\psi$ extends to an entire function by~\eqref{e:psi-supported}. Then:

1. If $A_1,A_2\in\Psi^{0}_{h,\nu}(M)$ satisfy
\begin{equation}
  \label{e:h-pseudolocality}
e^{tH_{p^2}}\big(\WFh(A_2)\big)\cap \WFh(A_1)=\emptyset\quad\text{for all }t\in [T_0,T_1],
\end{equation}
and at least one of $A_1,A_2$ is in $\Psi^{\comp}_{h,\nu}(M)$,
then $A_2 B(\omega) A_1=\mathcal O(h^\infty)_{\Psi^{-\infty}}$.

2. If additionally $\psi\in\mathscr S(\mathbb R)$ and $a\in S^0_{h,\nu}(M)$ is supported in an $h$-independent
compact subset of $T^*M$, then we have the Hilbert--Schmidt norm bound
with the constants depending only on $\psi$, some $S^0_{h,\nu}$
seminorm of $a$, and a fixed compact set containing $\supp a$,
\begin{equation}
  \label{e:HS-estimate}
\big\| \Op_h(a) B(\omega) \|_{\HS}^2,\big\|B(\omega)\Op_h(a)\|_{\HS}^2\leq Ch^{1-d}
\mu_L(S^*M\cap \supp a)+\O{}(h^\infty).
\end{equation}
\end{lemma}
\begin{proof}
We write $B(\omega)$ using the Fourier inversion formula:
$$
B(\omega)={1\over 2\pi}\int_{T_0}^{T_1}\widehat\psi(t)e^{-it\omega^2/h}e^{-ith\Delta_g}\,dt.
$$
Then~\eqref{e:h-pseudolocality} follows from the wavefront set properties
of the Schr\"odinger propagator $e^{-ith\Delta_g}$ (see for instance~\cite[Proposition~3.8]{DyGu}).
The estimate~\eqref{e:HS-estimate} follows from the proof of~\cite[Lemma~3.11]{DyGu}.
\end{proof}

\section{Reduction to the trapped set}
\label{s:reduction}

In this section we review the global properties of the scattering
resolvent and the half-wave propagator and prove several
statements which reduce the analysis to a neighborhood
of the trapped set $K$.

\subsection{Scattering resolvent}
  \label{s:resolvent-infinity}

The $L^2$ resolvent
$$
R_g(\lambda)=(-\Delta_g-\lambda^2)^{-1}:L^2(M)\to L^2(M),\quad
\Im\lambda>0
$$
admits a meromorphic continuation
$$
R_g(\lambda):L^2_{\comp}(M)\to L^2_{\loc}(M),\quad
\lambda\in\mathbb C\setminus (-\infty,0].
$$
In fact, when the dimension $d$ is odd, $R_g(\lambda)$ continues
meromorphically to $\lambda\in\mathbb C$,
and when $d$ is even, $R_g(\lambda)$ continues meromorphically
to the logarithmic cover of $\mathbb C$.
One way to prove meromorphic continuation is
by constructing an approximate inverse to $-\Delta_g-\lambda^2$ modulo 
a compact remainder which uses the free resolvent in $\mathbb R^d$~--
see for instance~\cite[\S 4.2]{ZwScat} or~\cite[Theorem~1.1]{SjoDist}.
(When $M$ has several infinite ends, we need to include the free
resolvent on each of these ends.) Another way is by using the method
of complex scaling which is reviewed below.

To study resonances in the region~\eqref{e:NRB}, we put $h:=R^{-1}$ and use
the semiclassically rescaled resolvent
$$
\mathcal R_g(\omega)=h^{-2}R_g(h^{-1}\omega),\quad
\omega\in\mathbb C\setminus (-\infty, 0],
$$
which is a right inverse to the operator $-h^2\Delta_g-\omega^2$.
For $\lambda=h^{-1}\omega$, the region in~\eqref{e:NRB} corresponds to
\begin{equation}
  \label{e:h-NRB-region}
\omega\ \in\ \Omega:= [1,1+h]+i[-\beta h,0].
\end{equation}
For resonance counting, it is convenient to prove estimates in a larger region,
\begin{equation}
  \label{e:big-Omega}
\widetilde\Omega:=[1-2h,1+2h]+i[-\tilde\beta h,2h],\quad
\tilde\beta > \beta.
\end{equation}
We next review the method of complex scaling,
following~\cite[\S4.3]{DyJAMS}.
 Fix small $\theta>0$ (the angle
of scaling) and $r_1>r_0$ (the place where scaling starts). Consider the following
totally real submanifold:
$$
\Gamma_\theta:=\Big\{y+if_\theta\big(|y|\big){y\over |y|}\ \colon\
y\in\mathbb R^d\Big\}\ \subset\ \mathbb C^d
$$
where $f_\theta\in C^\infty([0,\infty))$ is chosen so that
\begin{equation}
  \label{e:f-theta}
\begin{aligned}
f_\theta(r)=0,\quad
r\leq r_1;&\quad
f_\theta(r)=r\tan \theta,\quad
r\geq 2r_1;\\
f'_\theta(r)\geq 0,\quad
r\geq 0;&\quad
\{f'_\theta(r)=0\}=\{f_\theta(r)=0\}.
\end{aligned}
\end{equation}
Define the complex scaled differential operator $P_\theta$ on $M$ as follows:
\begin{itemize}
\item on $\{r<r_1\}$, $P_\theta$ is equal to $-h^2\Delta_g$;
\item on each infinite end of $M$ with Euclidean coordinate $y$,
$P_\theta$ is the restriction to $\Gamma_\theta$ (parametrized by $y$) of
the extension, $-h^2\sum_j \partial_{z_j}^2$, to $\mathbb C^n$ of the semiclassical Euclidean Laplacian $-h^2\Delta$. In polar coordinates $y=r\varphi$,
$$
P_\theta=\bigg({1\over 1+if'_\theta(r)}hD_r\bigg)^2
-{(d-1)i\over (r+if_\theta(r))(1+if'_\theta(r))}h^2D_r
-{h^2\Delta_\varphi\over(r+if_\theta(r))^2}
$$
with $\Delta_\varphi$ denoting Laplacian on the round sphere $\mathbb R^{d-1}$.
\end{itemize}
Then $P_\theta\in\Psi^2_h(M)$ is a second order semiclassical differential operator
on $M$ with principal symbol
$$
p_\theta:=\sigma_h(P_\theta)
$$
given by $p_\theta(x,\xi)=p(x,\xi)^2$ on $\{r<r_1\}$
and on each infinite end, in the polar coordinates $y=r\varphi$,
\begin{equation}
  \label{e:p-theta-inf}
p_\theta(r,\varphi,\eta_r,\eta_\varphi)={\eta_r^2\over (1+if'_\theta(r))^2}
+{|\eta_\varphi|^2\over (r+if_\theta(r))^2}.
\end{equation}
As shown for instance in~\cite[Theorems~4.36 and~4.38]{ZwScat} (whose proofs extend directly
to the case of several Euclidean ends),
for $h$ small enough so that $\widetilde\Omega\subset \{\Im(e^{i\theta}\omega)>0\}$
and all $s\in\mathbb R$
$$
P_\theta-\omega^2\text{ is a Fredholm operator of index zero }H^{s+2}(M)\to H^s(M),\quad
\omega\in\widetilde\Omega,
$$
and the poles of $(P_\theta-\omega^2)^{-1}$ in $\widetilde\Omega$
coincide with the poles of $\mathcal R_g(\omega)$, counted with multiplicities.

The next statement uses the structure of the complex scaled operator together with
propagation of singularities to show existence of a nontrapping parametrix
(see Figure~\ref{f:trapped-parametrix}):
\begin{lemma}
  \label{l:trapped-parametrix}
Assume that $Q\in\Psi^{\comp}_h(M)$ is supported inside $\{r<r_0\}$ and its
principal symbol is independent of $h$ and satisfies
\begin{equation}
  \label{e:trapped-parametrix-q}
\begin{gathered}
\sigma_h(Q)\geq 0\quad\text{everywhere};\\
\sigma_h(Q)>0\quad\text{on }K\cap S^*M.
\end{gathered}
\end{equation}
Then for $h$ small enough and $\omega\in\widetilde\Omega$,
the operator $P_\theta-iQ-\omega^2$ is invertible $H^2(M)\to L^2(M)$.
The inverse
\begin{equation}
  \label{e:trapped-parametrix-r}
\mathcal R_Q(\omega):=(P_\theta-iQ-\omega^2)^{-1}:L^2(M)\to H^2(M)
\end{equation}
is holomorphic and satisfies for each $s$
\begin{equation}
  \label{e:trapped-parametrix-est}
\|\mathcal R_Q(\omega)\|_{H^s_h(M)\to H^{s+2}_h(M)}\leq Ch^{-1}.
\end{equation}
Moreover, the operator $\mathcal R_Q(\omega)$ is \textbf{semiclassically
outgoing} in the sense that $A_2\mathcal R_Q(\omega)A_1=\mathcal O(h^\infty)_{\Psi^{-\infty}(M)}$
for all compactly supported $A_1,A_2\in\Psi^0_h(M)$ such that
\begin{equation}
  \label{e:semiout-cond}
\WFh(A_1)\cap\WFh(A_2)=e^{tH_p}\big(\WFh(A_1)\big)\cap \WFh(A_2) \cap S^*M=\emptyset\quad\text{for all }t\geq 0.
\end{equation}
\end{lemma}
\begin{figure}
\includegraphics{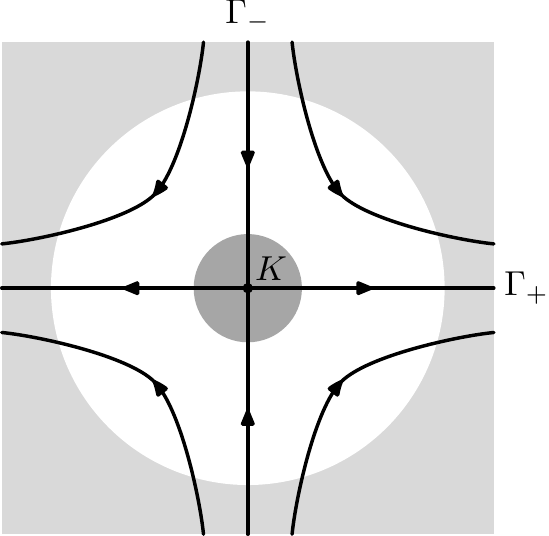}
\caption{An illustration of Lemma~\ref{l:trapped-parametrix}, showing trajectories
of $\varphi_t$ on $S^*M$. The shaded regions show places where $P_\theta-iQ-\omega^2$
is elliptic: the darker shaded region is $\{\sigma_h(Q)>0\}$
and the lighter shaded region is $\{f_\theta(r)\neq 0\}$.}
\label{f:trapped-parametrix}
\end{figure}
\begin{proof}
We follow~\cite[\S4.3]{DyJAMS}, see also~\cite[\S6.2.1]{ZwScat}.
We use semiclassical elliptic and propagation estimates 
for solutions to the equation
$$
\mathbf Pu=f\in H^s(M),\quad
u\in H^{s+2}(M)
$$
where
$$
\mathbf P:=P_\theta-iQ-\omega^2\in\Psi^2_h(M),\quad
\sigma_h(\mathbf P)=p_\theta-i\sigma_h(Q)-1.
$$
The operator $\mathbf P$ is elliptic for $r\geq 2r_1$, since
$$
\sigma_h(\mathbf P)(y,\eta)={|\eta|^2\over (1+i\tan\theta)^2}-1\quad\text{for }|y|\geq 2r_1.
$$
Moreover, $\mathbf P$ is elliptic near the fiber infinity of $M$, that is for
large enough $|\xi|$. By the elliptic estimate in the class $\Psi^2_h(M)$
(see for instance~\cite[Theorem~4.29]{EZB},
\cite[Proposition~2.4]{DyZeta}, or~\cite[\S E.2.2]{ZwScat})
there exists $\chi\in \Cc(M)$ such that for all $N$,
\begin{equation}
  \label{e:infes-ell}
\|(1-\chi)u\|_{H^{s+2}_h(M)}\leq C\|f\|_{H^s_h(M)}+\mathcal O(h^\infty)\|u\|_{H^{-N}_h(M)}.
\end{equation}
It remains to estimate $u$ in a compact set.
By~\eqref{e:f-theta} and~\eqref{e:p-theta-inf}
the operator $\mathbf P$ is elliptic
outside the set
$S^*M\cap \{f_\theta(r)=0\}\cap \{\sigma_h(Q)=0\}$. By the elliptic estimate,
we have for all $N$
\begin{equation}
  \label{e:infes-ell2}
\begin{gathered}
\|Bu\|_{H^{s+2}_h(M)}\leq C\|B'f\|_{H^s_h(M)}+\mathcal O(h^\infty)\|u\|_{H^{-N}_h(M)}\\
\text{for all compactly supported }B,B'\in\Psi^{0}_h(M)\text{ such that}\\
\WFh(B)\cap S^*M\cap \{f_\theta(r)=0\}\cap \{\sigma_h(Q)=0\}=\emptyset,\quad
\WFh(B)\subset\Ell_h(B').
\end{gathered}
\end{equation}
To estimate $\|Au\|$ for general $A$, we use the following statement:
for each $(x,\xi)\in T^*M$, there exists $T_{(x,\xi)}\geq 0$ such that
\begin{equation}
  \label{e:T-x-xi}
\exp(-T_{(x,\xi)}H_{\Re \sigma_h(\mathbf P)})(x,\xi)\ \notin\ S^*M\cap \{f_\theta(r)=0\}\cap \{\sigma_h(Q)=0\}. 
\end{equation}
Indeed, assume the contrary, and put $\gamma(t)=\exp(tH_{\Re\sigma_h(\mathbf P)})(x,\xi)$.
Clearly $(x,\xi)\in S^*M$. For all $t\leq 0$, we have $\gamma(t)\in \{f_\theta(r)=0\}$ and thus
(using that $f_\theta'(r)=f_\theta''(r)=0$ on $\{f_\theta(r)=0\}$)
$$
\gamma(t)=\exp(tH_{p^2})(x,\xi)
=\varphi_{2t}(x,\xi).
$$
Now, if $(x,\xi)\in\Gamma_+$, then
$\varphi_{-T}(x,\xi)\in \{\sigma_h(Q)>0\}$ for some $T>0$,
by~\eqref{e:trapprop-1} and~\eqref{e:trapped-parametrix-q}. If $(x,\xi)\notin\Gamma_+$,
then $\varphi_{-T}(x,\xi)\in \{r\geq 2r_1\}\subset \{f_\theta(r)\neq 0\}$
for some $T>0$. In either case we reach a contradiction,
finishing the proof of~\eqref{e:T-x-xi}.

By~\eqref{e:p-theta-inf} and~\eqref{e:trapped-parametrix-q},
\begin{equation}
  \label{e:imagination}
\Im\sigma_h(\mathbf P)\leq 0\quad\text{everywhere}.
\end{equation}
Using semiclassical propagation of singularities (see for instance~\cite[Theorem~E.49]{ZwScat}
or~\cite[Proposition~2.5]{DyZeta}) and~\eqref{e:infes-ell2}, we deduce that
\begin{equation}
  \label{e:infes-prop}
\begin{gathered}
\|Au\|_{H^{s+2}_h(M)}\leq Ch^{-1}\|A'f\|_{H^s_h(M)}+\mathcal O(h^\infty)\|u\|_{H^{-N}_h(M)}\\
\text{for all compactly supported }A,A'\in\Psi^{0}_h(M)\text{ such that }\WFh(A)\subset \Ell_h(A')\text{ and}\\
\varphi_{-2t}(x,\xi)\in\Ell_h(A')\text{ for all }(x,\xi)\in S^*M\cap \WFh(A),\
t\in [0,T_{(x,\xi)}].
\end{gathered}
\end{equation}
Indeed, by a pseudodifferential partition of unity we may reduce to the case
when $\WFh(A)$ is contained in a small neighborhood of some $(x,\xi)\in \overline T^*M$.
If $(x,\xi)\notin S^*M$, then we use~\eqref{e:infes-ell2}.
Otherwise we use propagation of singularities and~\eqref{e:T-x-xi}, \eqref{e:imagination},
and bound the term on the right-hand side of the propagation estimate by~\eqref{e:infes-ell2}.

Together~\eqref{e:infes-ell} and~\eqref{e:infes-prop} imply that
\begin{equation}
  \label{e:infes-all}
\|u\|_{H^{s+2}_h(M)}\leq Ch^{-1}\|\mathbf Pu\|_{H^s_h(M)}+\mathcal O(h^\infty)\|u\|_{H^{s+2}_h(M)}\quad\text{for all }
u\in H^{s+2}(M).
\end{equation}
As a compact perturbation of $P_\theta-\omega^2$,
$\mathbf P$ is a Fredholm operator $H^{s+2}(M)\to H^s(M)$, therefore~\eqref{e:infes-all}
implies that for $h$ small enough, $\mathbf P:H^{s+2}(M)\to H^s(M)$ is invertible and
\eqref{e:trapped-parametrix-est} holds. The restriction of the inverse to $C_c^\infty(M)$
does not depend on $s$.

It remains to show that under the condition~\eqref{e:semiout-cond},
we have $A_2\mathcal R_Q(\omega)A_1=\mathcal O(h^\infty)_{\Psi^{-\infty}(M)}$.
If $\WFh(A_1)\cap S^*M=\emptyset$ or $\WFh(A_2)\cap S^*M=\emptyset$, this follows from the elliptic estimate;
thus we may assume that $A_1,A_2\in\Psi^{\comp}_h(M)$.
Take $\tilde f\in H^{-N}(M)$ and put
$$
f:=A_1\tilde f,\quad
u:=\mathbf P^{-1}f.
$$
By~\eqref{e:semiout-cond}, we may find $A'\in\Psi^0_h(M)$ such that
$\WFh(A_1)\cap \WFh(A')=\emptyset$ and
\eqref{e:infes-prop} holds for $A:=A_2$ and~$A'$.
Then
$$
\|A_2u\|_{H^{s+2}_h(M)}\leq Ch^{-1}\|A'A_1\tilde f\|_{H^s_h(M)}+\mathcal O(h^\infty)\|u\|_{H^{-N}_h(M)}
=\mathcal O(h^\infty)\|\tilde f\|_{H^{-N}_h(M)},
$$
finishing the proof.
\end{proof}
We now prove two corollaries of Lemma~\ref{l:trapped-parametrix},
which in particular imply estimates on solutions to
\begin{equation}
  \label{e:equation-cs}
(P_\theta-\omega^2)u=f,\quad
u,f\in L^2(M),\quad
\omega\in\widetilde\Omega.
\end{equation}
The first statement implies that
$$
\|A_1u\|_{H^{s+2}_h(M)}\leq Ch^{-1}\|f\|_{H^s_h(M)}+\mathcal O(h^\infty)\|u\|_{H^{-N}_h(M)}\quad\text{when }
\WFh(A_1)\cap\Gamma^+\cap S^*M=\emptyset.
$$
\begin{lemma}
  \label{l:infinity-A1}
Assume that $A_1\in\Psi^0_{h,\nu}(M)$ is compactly supported and
$\WFh(A_1)\cap \Gamma^+\cap S^*M=\emptyset$. Then
there exists a neighborhood $U$ of $K\cap S^*M$
such that for all $Q$ satisfying~\eqref{e:trapped-parametrix-q} and $\WFh(Q)\subset U$,
we have
\begin{equation}
  \label{e:infinity-A1}
A_1\big(I-\mathcal R_Q(\omega)(P_\theta-\omega^2)\big)=\mathcal O(h^\infty)_{\Psi^{-\infty}},\quad
\omega\in\widetilde\Omega.
\end{equation}
\end{lemma}
\begin{proof}
Choose $U$ such that
$$
U\cap \WFh(A_1)=
U\cap \bigcup_{t\geq 0}\varphi_{-t}\big(\WFh(A_1)\cap S^*M\big)=\emptyset.
$$
This is possible by part~4 of Lemma~\ref{l:trapprop}.
Now
$$
A_1\big(I-\mathcal R_Q(\omega)(P_\theta-\omega^2)\big)=-iA_1\mathcal R_Q(\omega)Q
=
\mathcal O(h^\infty)_{\Psi^{-\infty}}
$$
by the semiclassically outgoing property in Lemma~\ref{l:trapped-parametrix}
(inserting an operator in $\Psi_h^0(M)$ between
$A_1$ and $\mathcal R_Q(\omega)$).
\end{proof}
The second corollary of Lemma~\ref{l:trapped-parametrix} implies the following bound for solutions of~\eqref{e:equation-cs}:
$$
\|u\|_{H^{s+2}_h}\leq C\|Bu\|_{H^s_h}+Ch^{-1}\|f\|_{H^s_h}+\mathcal O(h^\infty)\|u\|_{H^{-N}_h}\quad\text{when }
K\cap S^*M\subset\Ell_h(B).
$$
\begin{lemma}
  \label{l:infinity-B}
Assume that $B\in\Psi^0_h(M)$ is compactly supported and elliptic
on $K\cap S^*M$. Then for all $Q$
satisfying~\eqref{e:trapped-parametrix-q} and $\WFh(Q)\subset \Ell_h(B)$,
there exist $B_0,B_1,B_2\in\Psi^{\comp}_h(M)$ such that
\begin{equation}
  \label{e:infinity-B}
I=(B_1+h\mathcal R_Q(\omega)B_2)B+\mathcal R_Q(\omega)(I-B_0)(P_\theta-\omega^2)
+\mathcal O(h^\infty)_{\Psi^{-\infty}},\quad
\omega\in\widetilde\Omega.
\end{equation}
\end{lemma}
\begin{proof}
Take $B_0$ such that
$$
\WFh(Q)\cap \WFh(I-B_0)=\emptyset,\quad
\WFh(B_0)\subset\Ell_h(B).
$$
Then
$$
I-B_0=\mathcal R_Q(\omega)(P_\theta-\omega^2-iQ)(I-B_0)
$$
implies that
$$
I=B_0+\mathcal R_Q(\omega)(I-B_0)(P_\theta-\omega^2)-\mathcal R_Q(\omega)[P_\theta,B_0]
+\mathcal O(h^\infty)_{\Psi^{-\infty}}.
$$
It remains to use the elliptic parametrix construction to find $B_1$, $B_2$ so that 
$$
B_2B=-h^{-1}[P_\theta,B_0]+\O{}(h^\infty)_{\Psi^{-\infty}},\quad
 B_1B=B_0+\O{}(h^\infty)_{\Psi^{-\infty}}
$$
and~\eqref{e:infinity-B} follows.
\end{proof}
The next statement, which is an important technical tool in the
construction of the approximate inverse in~\S\ref{s:weyl-inverse},
is obtained by iteration of Lemmas~\ref{l:basic-parametrix}
and~\ref{l:infinity-A1}. See Figure~\ref{f:iterated-parametrix}.
\begin{lemma}
  \label{l:iterated-parametrix}
Fix $\nu\in [0,1/2)$ and assume that a sequence of symbols
$$
a_j\in S^{0}_{h,\nu}(T^*M),\quad
j=0,1,\dots,L=L(h),\quad
0<L(h)\leq C\log(1/h)
$$ 
is supported in a fixed compact subset $W\subset T^*M\setminus 0$ and each $S^0_{h,\nu}$
seminorm of $a_j$ is bounded uniformly in~$j$. Assume moreover
that $|a_j|\leq 1$ and 
there exists an $h$-independent open neighborhood $V$ of $\Gamma_+\cap S^*M$
and there exists $t_1>0$ bounded independently of $h$
such that the following dynamical conditions hold for all $j$:
\begin{gather}
  \label{e:dyncon}
\varphi_{-t_1}(\supp a_j)\cap \supp(1-a_{j+1})\cap V=\emptyset\quad\text{for all }
j=0,\dots,L-1,\\
  \label{e:dyncon-2}
\varphi_{-t}(W)\subset \{r<r_1\}\quad\text{for all }t\in [0,t_1].
\end{gather}
Then we have for all $\omega\in\widetilde\Omega$, on $H^2(M)$
\begin{equation}
  \label{e:iterated-parametrix}
\Op_h(a_0)=Z(\omega)(P_\theta-\omega^2)
+J(\omega)\Op_h(a_L)+\mathcal O(h^\infty)_{\Psi^{-\infty}}
\end{equation}
where $Z(\omega):L^2(M)\to H^2(M)$, $J(\omega):H^{-N}(M)\to H^{N}(M)$
are holomorphic in $\omega\in\widetilde\Omega$ and satisfy the bounds
for each $\varepsilon_1>0$
\begin{align}
  \label{e:iter-bounds-Z}
\|Z(\omega)\|_{H^s_h\to H^{s+2}_h}&\leq C_{s,\varepsilon_1}h^{-1}\exp\big((\tilde\beta t_1+\varepsilon_1) L\big),\\
  \label{e:iter-bounds-J}
\|J(\omega)\|_{H^{-N}_h\to H^{N}_h}&\leq C_{N,\varepsilon_1}\exp\Big(\Big(-{\Im\omega\over h}t_1+\varepsilon_1\Big)L\Big).
\end{align}
Finally, if $a_0=1$ on some $h$-independent neighborhood of $K\cap S^*M$, then
a decomposition of the form~\eqref{e:iterated-parametrix}
holds with $\Op_h(a_0)$ replaced by the identity operator.
\end{lemma}
\begin{figure}
\includegraphics{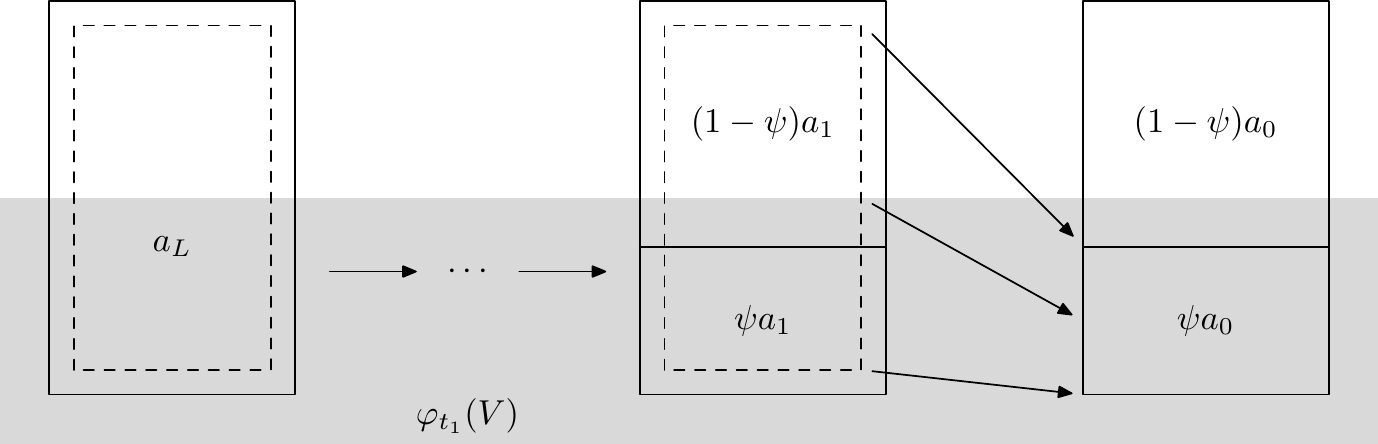}
\caption{An illustration of Lemma~\ref{l:iterated-parametrix}, showing
the supports of $\psi a_j$, $(1-\psi) a_j$, and $(\psi a_j)\circ\varphi_{t_1}$ (dashed),
as well as $\varphi_{t_1}(V)$ (shaded). The arrows correspond to $\varphi_{t_1}$.
At each step of the iteration, $(1-\psi)a_j$ is expressed using Lemma~\ref{l:infinity-A1}
and $\psi a_j$ is reduced to $a_{j+1}$ using Lemma~\ref{l:basic-parametrix}.}
\label{f:iterated-parametrix}
\end{figure}
\begin{proof}
Fix $h$-independent $\psi\in C_c^\infty(\varphi_{t_1}(V);[0,1])$ such that
$$
\supp(1-\psi)\cap \Gamma_+\cap S^*M\cap W=\emptyset.
$$
Then $\supp((1-\psi)a_j)$ is contained in an $h$-independent compact
subset of $T^*M$ not intersecting $\Gamma_+\cap S^*M$, thus by Lemma~\ref{l:infinity-A1}
for an appropriate choice of $Q$ we have
for $j=0,\dots,L-1$
\begin{equation}
  \label{e:steppi-1}
\Op_h\big((1-\psi)a_j\big)=\Op_h\big((1-\psi)a_j\big)\mathcal R_{Q}(\omega)(P_\theta-\omega^2)
+\mathcal O(h^\infty)_{\Psi^{-\infty}}.
\end{equation}
Next, by~\eqref{e:dyncon} we have
$$
\varphi_{-t_1}(\supp(\psi a_j))\cap \supp(1-a_{j+1})=\emptyset.
$$
Using~\eqref{e:dyncon-2},
fix a multiplication operator $B'=B'(x)\in C_c^\infty(M;[0,1])$ such that
$$
\supp B'\subset \{r<r_1\},\quad
\supp(1-B')\cap \bigcup_{t=0}^{t_1}\varphi_{-t}(W)=\emptyset.
$$
Since $P_\theta=-h^2\Delta_g$ on $\{r<r_1\}$, we have
$B'(P_\theta-\omega^2)=B'(-h^2\Delta_g-\omega^2)$.
Therefore by Lemma~\ref{l:basic-parametrix}, 
\begin{equation}
  \label{e:steppi-2}
\Op_h(\psi a_j)=Z_j(\omega)B'(P_\theta-\omega^2)+e^{i\omega t_1/h}\Op_h(\psi a_j)
U(t_1)\Op_h(a_{j+1})+\mathcal O(h^\infty)_{\Psi^{-\infty}}
\end{equation}
for all $\omega\in\widetilde\Omega$,
where $Z_j(\omega)$ is holomorphic in $\omega\in\widetilde\Omega$ and satisfies
$$
\|Z_j(\omega)\|_{H^{-N}_h\to H^N_h}\leq C_Nh^{-1}
$$
and the constant $C_N$, as well as the constants in $\mathcal O(h^\infty)_{\Psi^{-\infty}}$, is independent of $h$ and $j$.

Adding~\eqref{e:steppi-1} and~\eqref{e:steppi-2} and iterating in $j$, we obtain~\eqref{e:iterated-parametrix} with
$$
\begin{aligned}
Z(\omega)&=\sum_{j=0}^{L-1}e^{i\omega jt_1/h}\Big(\prod_{\ell=0}^{j-1}\Op_h(\psi a_\ell)U(t_1)\Big)\big(
\Op_h\big((1-\psi)a_j\big)\mathcal R_{Q}(\omega)+Z_j(\omega)B'\big),\\
J(\omega)&=e^{i\omega Lt_1/h}\prod_{j=0}^{L-1}\Op_h(\psi a_j)U(t_1).
\end{aligned}
$$
The bounds~\eqref{e:iter-bounds-Z} and~\eqref{e:iter-bounds-J}
follow from here and estimate on the operator norm following from~\eqref{e:opernorm}:
$$
\max_j \|\Op_h(\psi a_j)\|_{L^2\to L^2}\leq 1+o(1)\quad\text{as }h\to 0.
$$
In particular, for any fixed $\varepsilon_1>0$ we have
$$
\max_{0\leq j\leq L}\Big\|\prod_{\ell=0}^{j-1}\Op_h(\psi a_\ell)U(t_1)\Big\|_{H^{-N}_h\to H^N_h}
\leq C_Ne^{\varepsilon_1 L}.
$$
To show the last statement of the lemma, assume that $a_0=1$ on an $h$-independent
neighborhood $\mathcal U$ of $K\cap S^*M$. Take $B\in\Psi^{\comp}_h(M)$ elliptic on $K\cap S^*M$
and satisfying $\WFh(B)\subset\mathcal U$. Then by Lemma~\ref{l:infinity-B},
we have for an appropriate choice of $Q,B_0,B_1,B_2\in\Psi^{\comp}_h(M)$,
$$
I=\mathcal R_Q(\omega)(I-B_0)(P_\theta-\omega^2)+(B_1+h\mathcal R_Q(\omega)B_2)B\Op_h(a_0)+\mathcal O(h^\infty)_{\Psi^{-\infty}}.
$$
Combining this with the representation~\eqref{e:iterated-parametrix} of $\Op_h(a_0)$,
we obtain~\eqref{e:iterated-parametrix} with the identity operator on the left-hand side.
\end{proof}

\subsection{Wave propagator}
  \label{s:reduce-wave}

We next study the long time behavior of the half-wave propagator $U(t)=\exp(-it\sqrt{-\Delta_g})$.
We first prove a microlocal estimate on the free half-wave propagator on $\mathbb R^d$,
$$
U_0(t)=\exp(-it\sqrt{-\Delta_0}):\ L^2(\mathbb R^d)\to L^2(\mathbb R^d),
$$
where $\Delta_0$ is the flat Laplacian.
\begin{lemma}
\label{l:EuclideanOutgoing}
Let $A_1,A_2\in \Psi^k_h(\mathbb R^d)$ such that there exists $R>0$ with 
$$
\WFh(A_1)\cup \WFh(A_2)\ \subset\ \{|y|<R\},
$$
 at least one of $\WFh(A_1)$, $\WFh(A_2)$  is a compact subset of $T^*\mathbb R^d\setminus 0$,
and
\begin{equation}
  \label{e:eo-cond}
(y',\eta)\in\WFh(A_1),\
\eta\neq 0,\
t\geq 0\ \Longrightarrow\
\Big(y'+t{\eta\over|\eta|},\eta\Big)\notin\WFh(A_2).
\end{equation}
Then we have the following version of propagation of singularities which is uniform in $t\geq 0$: 
\begin{equation}
  \label{e:eo-conc}
A_2 U_0(t) A_1=\mathcal O(h^\infty)_{\Psi^{-\infty}(\mathbb R^d)}.
\end{equation}
\end{lemma}
\begin{proof}
Write $A_1=\Op_h(a_1)^*+\mathcal O(h^\infty)_{\Psi^{-\infty}}$,
$A_2=\Op_h(a_2)+\mathcal O(h^\infty)_{\Psi^{-\infty}}$ for some
$a_1,a_2$ whose supports satisfy the conditions
imposed on $\WFh(A_1)$, $\WFh(A_2)$, including~\eqref{e:eo-cond}.
The Schwartz kernel of $\Op_h(a_2) U_0(t)\Op_h(a_1)^*$ is compactly supported and given by 
\begin{gather} 
\label{e:kernel}
\mathcal K(y,y')=(2\pi h)^{-2d}\int_{\mathbb R^d} e^{\frac{i}{h}(\la y-y',\eta\ra -t|\eta|)}a_2(y,\eta)\overline{a_1(y',\eta)}\,d\eta.
\end{gather}
Put $\Phi=\langle y-y',\eta\rangle-t|\eta|$. Then there exists $c>0$ such that on the support of $a_2(y,\eta)\overline{a_1(y',\eta)}$,
\begin{equation}
  \label{e:hula}
|\partial_\eta\Phi|=
\Big|y-y'-t\frac{\eta}{|\eta|}\Big|\geq c\la t\ra>0.
\end{equation}
Indeed, since $y,y'$ vary in a compact set and $\eta$ is bounded
away from zero, it is enough to
consider the case of bounded $t$. Then~\eqref{e:hula} follows from~\eqref{e:eo-cond}.

Now, repeated integration by parts in $\eta$ gives that for each $N$,
$$
\|\mathcal K\|_{C^N(\mathbb R^{2d})}\ \leq\ C_Nh^N\langle t\rangle^{-N}.
$$
This completes the proof.
\end{proof}
We next use $U_0(t)$ to write a parametrix for the propagator $U(t)$.
For $\psi_0\in \Cc(M)$ with $\supp (1-\psi_0)\subset \{r> r_0\}$
and $u\in L^2(M)$,
we define
$$
(1-\psi_0)U_0(t)(1-\psi_0)u\in L^2(M)
$$
as follows: we pull back the restriction of $(1-\psi_0)u$ to each infinite end
to $\mathbb R^d$ using the Euclidean coordinate, apply $(1-\psi_0) U_0(t)$,
and take the sum of the resulting functions pulled back to $M$. This gives
an operator
\begin{equation}
  \label{e:pushed-operator}
(1-\psi_0)U_0(t)(1-\psi_0):L^2(M)\to L^2(M).
\end{equation}
Recall the sets
$\mathcal E_{\pm},\mathcal E_{\pm}^\circ$
defined in~\eqref{e:de-sets-M}.
\begin{lemma}
\label{l:comparePropagator}
Suppose that $A_{\pm}\in \Ph{\comp}{}(M)$, $\psi_0\in\Cc(M)$ satisfy for some $r_2>r_0$
(see Figure~\ref{f:comparePropagator})
$$
\WFh(A_\pm)\subset\mc{E}_{\pm}^\circ\cap \{r> r_2\},\quad
\supp\psi_0\subset\{r<r_2\},\quad
\supp(1-\psi_0)\subset \{r>r_0\}.
$$
Then we have uniformly in $0\leq t\leq C h^{-1}$
\begin{gather}
\label{e:forward}
U(t)A_+=(1-\psi_0)U_0(t)(1-\psi_0)A_++\O{}(h^\infty)_{\Psi^{-\infty}}\\
\label{e:backward}
U(-t)A_-=(1-\psi_0)U_0(-t)(1-\psi_0)A_-+\O{}(h^\infty)_{\Psi^{-\infty}}.
\end{gather}
\end{lemma}
\begin{figure}
\includegraphics{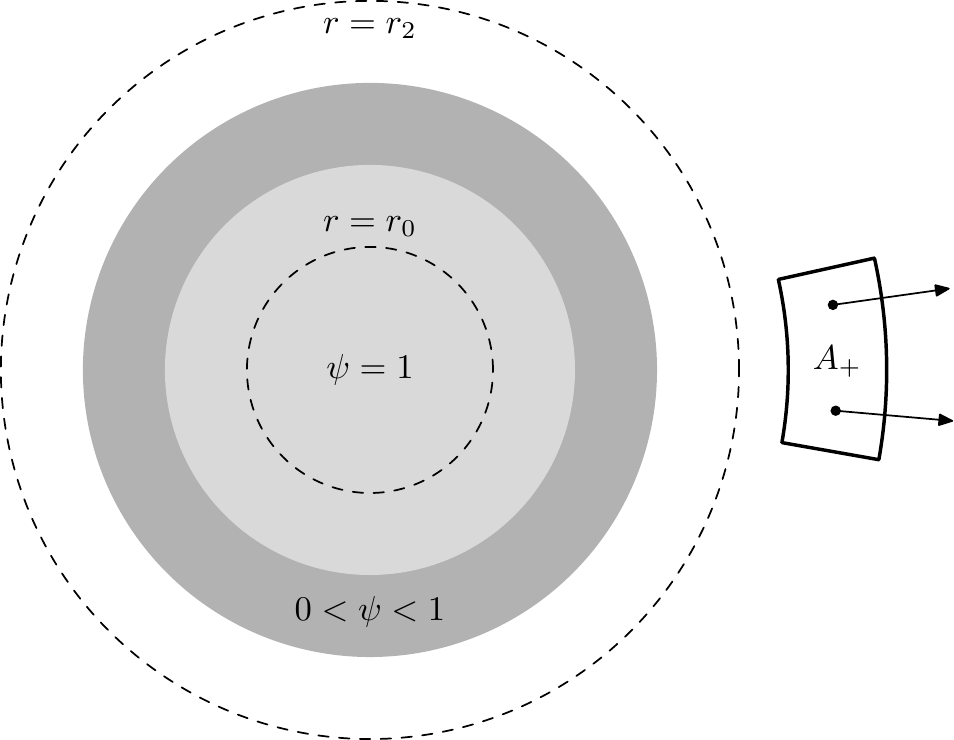}
\caption{An illustration of Lemma~\ref{l:comparePropagator} when $0\leq \psi_0\leq 1$, showing
the regions $\psi_0=1$ and $0<\psi_0<1$ (shaded) and the projection
of $\WFh(A_+)$ onto $M$. The points $(x,\xi)$ in $\WFh(A)$, pictured by arrows,
give rise to trajectories escaping to infinity in the future and
never entering $\supp\psi_0$.}
\label{f:comparePropagator}
\end{figure}
\begin{proof}
We prove~\eqref{e:forward}, with~\eqref{e:backward} established similarly.
For simplicity of notation, we present the argument in the case when $M$ is diffeomorphic to $\mathbb R^d$.
The general case is proved in the same way, reducing to the case when $A_+$ is supported
on one infinite end and treating $1-\psi_0$ on this infinite end as an operator $L^2(M)\to  L^2(\mathbb R^d)$
and $L^2(\mathbb R^d)\to L^2(M)$.
We identify $M$ with $\mathbb R^d$ and use the quantization~\eqref{e:standard-op}.

Since $U_0(t),U(t)$ are bounded uniformly in $t$ on all Sobolev spaces
and $\WFh(A_+)\cap\supp\psi_0=\emptyset$,
$$
\begin{aligned}
U(t)A_+&=U(t)(1-\psi_0)^2A_++\mathcal O(h^\infty)_{\Psi^{-\infty}}.
\end{aligned}
$$
Therefore it remains to show that uniformly in $0\leq t\leq Ch^{-1}$,
\begin{equation}
  \label{e:cornwall2}
W(t)=\mathcal O(h^\infty)_{\Psi^{-\infty}},
\end{equation}
where the operator $W(t)$ on $L^2(M)$ is defined by
$$
W(t):=\big((1-\psi_0)U_0(t)-U(t)(1-\psi_0)\big)(1-\psi_0)A_+.
$$
Using the wave operator $\Box_g=\partial_t^2-\Delta_g$, we write
\begin{equation}
  \label{e:Wrep}
W(t)=\cos(t\sqrt{-\Delta_g})W(0)+{\sin(t\sqrt{-\Delta_g})\over \sqrt{-\Delta_g}}W'(0)+\int_0^t {\sin\big((t-t')\sqrt{-\Delta_g}\big)\over \sqrt{-\Delta_g}}\Box_g W(t')\,dt'.
\end{equation}
We compute
\begin{equation}
  \label{e:ilya0}
W(0)=0.
\end{equation}
Next,
\begin{equation}
  \label{e:ilya1}
ihW'(0)=\big((1-\psi_0)h\sqrt{-\Delta_0}-h\sqrt{-\Delta_g}(1-\psi_0)\big)(1-\psi_0)A_+=\mathcal O(h^\infty)_{\Psi^{-\infty}}.
\end{equation}
Indeed, by~\eqref{e:sqrt} both
$(1-\psi_0)h\sqrt{-\Delta_0}(1-\psi_0)A_+$ and
$h\sqrt{-\Delta_g}(1-\psi_0)^2A_+$ are in $\Psi^0_h(M)$.
As explained in the discussion following~\cite[Theorem~8.7]{d-s},
the asymptotic expansion for the full symbol of each of these operators
at some point can be computed using only the derivatives of
$\psi_0$ and the full symbols of $A_+, \Delta_0,\Delta_g$
at this point. Since $\Delta_0=\Delta_g$ and $\psi_0=0$ on $\{r>r_2\}\supset\WFh(A_+)$,
we obtain~\eqref{e:ilya1}.

Finally, since $\Delta_0=\Delta_g$ on $\{r>r_0\}\supset\supp(1-\psi_0)$, we have
$$
h^2\Box_g W(t)=[h^2\Delta_g,\psi_0]U_0(t)(1-\psi_0)A_+.
$$
Now, with $A_2:=[h^2\Delta_g,\psi_0]$
$$
\WFh(A_2)\ \subset\ \supp d\psi_0\ \subset\ \{r_0<r<r_2\}.
$$
Then $A_2$ and $A_1:=A_+$ satisfy~\eqref{e:eo-cond},
thus by Lemma~\ref{l:EuclideanOutgoing}
\begin{equation}
  \label{e:ilya2}
h^2\Box_g W(t)=\mathcal O(h^\infty)_{\Psi^{-\infty}}.
\end{equation}
Now~\eqref{e:cornwall2} follows from~\eqref{e:Wrep}--\eqref{e:ilya2}, the bound
$t\leq Ch^{-1}$,
and the fact that for each $s$, the operators
$$
\cos(t\sqrt{-\Delta_g}),
{\sin(t\sqrt{-\Delta_g})\over\sqrt{-\Delta_g}}:H^s_h(M)\to H^s_h(M)
$$
are bounded in norm by $C\langle t\rangle$.
\end{proof}
The next lemma shows that for times $t=\mathcal O(\log(1/h))$, the cutoff wave propagator
$A_2U(t)A_1$, where $A_j\in\Psi^{\comp}_{h,\nu}(T^*M)$
and $\WFh(A_j)$ lies near $S^*M$, can be expressed
in terms of cutoff wave propagators for bounded time.
It relies on Lemmas~\ref{l:EuclideanOutgoing} and~\ref{l:comparePropagator}
and is a key component of the proof of Lemma~\ref{l:propLong} below.
\begin{lemma}
  \label{l:super-sandwich}
Let $A_1\in\Psi^{\comp}_{h,\nu}(M)$,
$A_2\in\Psi^0_{h,\nu}(M)$, and
$\chi\in S^0_h(T^*M;[0,1])$ satisfy
for some $\varepsilon_E>0$ and~$r_2>r_0$
\begin{gather}
  \label{e:chi-sandwich-1}
\WFh(A_1)\cup\WFh(A_2)\cup\supp\chi\subset \{r<r_2\},\quad
\WFh(A_1)\subset \{|\xi|_g^2\in (1-\varepsilon_E,1+\varepsilon_E)\},
\\
  \label{e:chi-sandwich-2}
\supp(1-\chi)\cap \{|\xi|_g^2\in [1-\varepsilon_E,1+\varepsilon_E]\}\cap \{r\leq r_0\}=\emptyset.
\end{gather}
Put
$T:=\sqrt{r_2^2-r_0^2}$
and let $C$ be an $h$-independent constant.
Then for each sequence of times
$$
t_1,\dots,t_L\geq T,\quad
L\leq Ch^{-1},\quad
t_j\leq C,
$$
we have
$$
A_2 U(t_1+\dots+t_L)A_1
=A_2 U(t_1)\Op_h(\chi) U(t_2)\cdots\Op_h(\chi) U(t_L)A_1+\mathcal O(h^\infty)_{\Psi^{-\infty}}.
$$
\end{lemma}
\begin{proof}
We may assume that $A_1\in\Psi^{\comp}_h(M)$, $A_2\in\Psi^0_h(M)$. Indeed, otherwise we
may take $A'_1\in\Psi^{\comp}_h(M)$,
$A'_2\in\Psi^0_h(M)$ such that
$(I-A'_1)A_1=\mathcal O(h^\infty)_{\Psi^{-\infty}}$,
$A_2(I-A'_2)=\mathcal O(h^\infty)_{\Psi^{-\infty}}$,
and $\WFh(A'_1)$, $\WFh(A'_2)$ satisfy~\eqref{e:chi-sandwich-1},
and apply the argument below with $A_1,A_2$
replaced by $A'_1,A'_2$.

We have
$$
\begin{gathered}
A_2U(t_1+\dots+t_L)A_1
-A_2 U(t_1)\Op_h(\chi) U(t_2)\cdots\Op_h(\chi) U(t_L)A_1=\sum_{\ell=1}^{L-1} B_\ell,\\
B_\ell:=A_2U(t_1)\Op_h(\chi)\cdots U(t_{\ell-1})\Op_h(\chi)U(t_\ell)\Op_h(1-\chi)U(t_{\ell+1}+\dots+t_L)A_1.
\end{gathered}
$$
Therefore it suffices to show that $B_\ell=\mathcal O(h^\infty)_{L^2\to L^2}$
uniformly in $\ell$. Since $U(t)$ is unitary and $\Op_h(\chi)$
satisfies the norm bound~\cite[Theorem~13.13]{EZB}
\begin{equation}
  \label{e:chi-bdddd}
\|\Op_h(\chi)\|_{H^s_h\to H^s_h}\leq 1+\mathcal O(h),
\end{equation}
it is enough to show the following bounds uniform in $\ell$
(in fact~\eqref{e:roo-1} is used only for $\ell=2,\dots,L-1$
and~\eqref{e:roo-2} is used only for $\ell=1$)
\begin{align}
  \label{e:roo-1}
\Op_h(\chi)U(t_\ell)\Op_h(1-\chi)U(t_{\ell+1}+\dots+t_L)A_1=\mathcal O(h^\infty)_{\Psi^{-\infty}},\\
  \label{e:roo-2}
A_2U(t_\ell)\Op_h(1-\chi)U(t_{\ell+1}+\dots+t_L)A_1=\mathcal O(h^\infty)_{\Psi^{-\infty}}.
\end{align}
We show~\eqref{e:roo-1} with the same proof giving~\eqref{e:roo-2} as well.
Take $\psi_1\in C_c^\infty(\mathbb R)$ such that
$$
\supp\psi_1\subset (1-\varepsilon_E,1+\varepsilon_E),\quad
\WFh(A_1)\cap \supp\big(1-\psi_1(|\xi|_g^2)\big)=\emptyset.
$$
We can replace $A_1$ by $\psi_1(-h^2\Delta_g)A_1$ in~\eqref{e:roo-1} since
$$
(I-\psi_1(-h^2\Delta_g))A_1=\mathcal O(h^\infty)_{\Psi^{-\infty}}.
$$
Since $U(t_{\ell+2}+\dots+t_L)$ commutes with $\psi_1(-h^2\Delta_g)$, it suffices to show that
\begin{equation}
  \label{e:hanger-1}
AU(t_{\ell+2}+\dots+t_L)A_1=\mathcal O(h^\infty)_{\Psi^{-\infty}},
\end{equation}
where
$$
A:=U(-t_\ell-t_{\ell+1})\Op_h(\chi)U(t_\ell)\Op_h(1-\chi)U(t_{\ell+1})\psi_1(-h^2\Delta_g).
$$
By Lemma~\ref{l:egorov}, we have $A\in \Psi^{\comp}_h(M)$ and
$$
\WFh(A)\subset \varphi_{-t_\ell-t_{\ell+1}}(\supp\chi)\cap \varphi_{-t_{\ell+1}}(\supp(1-\chi))\cap \{|\xi|^2_g\in (1-\varepsilon_E,1+\varepsilon_E)\}.
$$
Take $\mathbf x\in \varphi_{t_{\ell+1}}(\WFh(A))$. By~\eqref{e:chi-sandwich-2}
we have $\mathbf x\in \{r>r_0\}$
and by~\eqref{e:chi-sandwich-1} we have
$\varphi_{t_\ell}(\mathbf x)\in \{r<r_2\}$. By~\eqref{e:convexinf-1}
and since $t_\ell\geq T$
we see that $\mathbf x\in \mathcal E_-^\circ$.
Applying~\eqref{e:convexinf-1} again and using that $t_{\ell+1}\geq T$
we see that $\varphi_{-t_{\ell+1}-s}(\mathbf x)\in \mathcal E_-^\circ\cap \{r>r_2\}$
for all $s\geq 0$. Therefore
\begin{gather}
  \label{e:why-panda-1}
\WFh(A)\subset \mathcal E_-^\circ\cap \{r>r_2\} ,\\
  \label{e:why-panda-2}
\varphi_{-s}(\WFh(A))\cap \WFh(A_1)=\emptyset\quad\text{for all }s\geq 0.
\end{gather}
Denote $\tilde t_\ell:=t_{\ell+2}+\dots+t_L\in [0,Ch^{-1}]$.
By~\eqref{e:why-panda-1} we may
apply Lemma~\ref{l:comparePropagator} to get for some $\psi_0\in C_c^\infty(M;\mathbb R)$, $\supp(1-\psi_0)\subset \{r> r_0\}$
$$
U(-\tilde t_\ell)A^*=(1-\psi_0)U_0(-\tilde t_\ell)(1-\psi_0)A^*+\mathcal O(h^\infty)_{\Psi^{-\infty}}.
$$
Taking adjoints, we get
\begin{equation}
  \label{e:hanger-2}
AU(\tilde t_\ell)=A(1-\psi_0)U_0(\tilde t_\ell)(1-\psi_0)+\mathcal O(h^\infty)_{\Psi^{-\infty}}.
\end{equation}
By Lemma~\ref{l:EuclideanOutgoing} and~\eqref{e:why-panda-2} we have
\begin{equation}
  \label{e:hanger-3}
A(1-\psi_0)U_0(\tilde t_\ell)(1-\psi_0)A_1=\mathcal O(h^\infty)_{\Psi^{-\infty}}.
\end{equation}
Combining~\eqref{e:hanger-2} and~\eqref{e:hanger-3}, we
obtain~\eqref{e:hanger-1}, finishing the proof.
\end{proof}
Using Lemma~\ref{l:super-sandwich}, we also obtain the following
estimate used in~\S\ref{s:proof-of-decay}:
\begin{lemma}
  \label{l:mega-sandwich}
Assume that $A_1\in \Psi^{\comp}_h(M),A_2\in\Psi^0_h(M)$ satisfy for some $r_1>r_0$
and $\varepsilon_E>0$
\begin{equation}
  \label{e:mes-1}
\WFh(A_1)\subset \{r<r_1\}\cap \{|\xi|_g^2\in(1-\varepsilon_E,1+\varepsilon_E)\},\quad
\WFh(A_2)\subset \{r<r_1\}.
\end{equation}
Put $T_0:=\sqrt{r_1^2-r_0^2}$ and assume that $\chi'\in C_c^\infty(M)$ satisfies
\begin{equation}
  \label{e:mes-2}
\supp(1-\chi')\cap \{r\leq r_1+T_0\}=\emptyset.
\end{equation}
Fix $C_0>0$.
Then for all $t\in [T_0,C_0h^{-1}]$,
$s\in [0,C_0h^{-1}]$, and $u\in L^2(M)$ we have
$$
\|A_2U(s+t)A_1u\|_{L^2}
\leq \|A_2\|_{L^2\to L^2}\cdot\|\chi' U(t)A_1u\|_{L^2}
+\mathcal O(h^\infty)\|u\|_{L^2}.
$$
\end{lemma}
\begin{proof}
We first consider the case $s\geq T_0$. Fix $\chi\in C_c^\infty(M;[0,1])$ such that
$$
\supp\chi\subset \{r<r_1\},\quad
\supp(1-\chi)\cap \{r\leq r_0\}=\emptyset.
$$
We write
$$
t=t_1+\dots+t_L,\quad
s=s_1+\dots+s_{L'},\quad
t_j,s_j\in [T_0,2T_0],\quad
L,L'\leq C_0h^{-1}.
$$
By Lemma~\ref{l:super-sandwich} (with $(r_1,T_0)$ taking the place of $(r_2,T)$) we have
$$
A_2 U(s+t)A_1=A_2 U(s_1)\chi
\cdots U(s_{L'})\chi U(t_1)\cdots \chi
U(t_L)A_1+\mathcal O(h^\infty)_{\Psi^{-\infty}}.
$$
Therefore
$$
\|A_2 U(s+t)A_1u\|_{L^2}\leq \|A_2\|_{L^2\to L^2}\cdot\|\chi U(t_1)\cdots \chi
U(t_L)A_1u\|_{L^2}+\mathcal O(h^\infty)\|u\|_{L^2}.
$$
Another application of Lemma~\ref{l:super-sandwich} gives
$$
\|\chi U(t_1)\cdots \chi
U(t_L)A_1u-\chi U(t)A_1u\|_{L^2}=\mathcal O(h^\infty)\|u\|_{L^2},
$$
finishing the proof since $\chi=\chi\chi'$.

We now consider the case $0\leq s\leq T_0$. Fix
$\psi_1\in C_c^\infty(\mathbb R;[0,1])$ such that
$\supp\psi_1\subset (0,\infty)$ and
$\supp(1-\psi_1)\cap [1-\varepsilon_E,1+\varepsilon_E]=\emptyset$.
Since $U(t)$ commutes with $\psi_1(-h^2\Delta_g)$, we have
$$
\begin{aligned}
A_2 U(s+t)A_1&=A_2 U(s+t)\psi_1(-h^2\Delta_g)A_1
+\mathcal O(h^\infty)_{\Psi^{-\infty}}\\
&=
A_2 U(s)\psi_1(-h^2\Delta_g)U(t)A_1+\mathcal O(h^\infty)_{\Psi^{-\infty}}.
\end{aligned}
$$
Therefore
$$
\|A_2U(s+t)A_1u\|_{L^2}\leq \|U(-s)A_2 U(s)\psi_1(-h^2\Delta_g)U(t)A_1u\|_{L^2}
+\mathcal O(h^\infty)\|u\|_{L^2}.
$$
By~\eqref{e:mes-1} and \eqref{e:mes-2} we have
$(T^*M\setminus 0)\cap\varphi_{-s}(\WFh(A_2))\cap \supp(1-\chi')=\emptyset$.
Therefore by Lemma~\ref{l:egorov}
$$
U(-s)A_2 U(s)\psi_1(-h^2\Delta_g)(1-\chi')=\mathcal O(h^\infty)_{\Psi^{-\infty}}.
$$
Therefore
$$
\|U(-s)A_2 U(s)\psi_1(-h^2\Delta_g)U(t)A_1u\|_{L^2}
\leq \|A_2\|_{L^2\to L^2}\cdot \|\chi'U(t)A_1u\|_{L^2}+\mathcal O(h^\infty)\|u\|_{L^2}
$$
finishing the proof.
\end{proof}

\section{Dynamical cutoff functions}
  \label{s:cutoffs}

In this section, we construct families of auxiliary cutoff functions which
localize to smaller and smaller neighborhoods of $\Gamma_\pm$ and
are the key component of the proofs of Theorems~\ref{thm:fractalWeyl} and~\ref{thm:decayOnAverage}.
These functions are defined by propagating a fixed cutoff function
for a large time.

Fix constants 
$$
0\leq \rho<2\nu <1.
$$
We propagate up to time $\rho t_e$
where $t_e$ is the Ehrenfest time from~\eqref{e:t-e-R}
in the semiclassical scaling:
\begin{equation}
  \label{e:t-e-h}
t_e={\log (1/h)\over 2\Lambda_{\max}}.
\end{equation}
Fix a cutoff function
\begin{equation}
\label{e:chiprop1}
\chi\in \Cc(T^*M\setminus 0;[0,1]),\quad
\supp(1-\chi)\cap K\cap S^*M=\emptyset.
\end{equation}
Define the following functions
living near $\Gamma_\pm$:
\begin{equation}
  \label{e:chi-pm}
\chi_t^+=\chi(\chi\circ \varphi_{-t}),\quad
\chi_t^-=\chi(\chi\circ \varphi_t),\quad
t\geq 0.
\end{equation}
By the derivative estimates for the flow $\varphi_t$ (see for instance~\cite[Lemma~C.1]{DyGu2}) we have uniformly in $t$,
\begin{equation}
  \label{e:symbols-good}
\chi_t^\pm\in S^{\comp}_{h,\nu}(T^*M),\quad
0\leq t\leq \rho t_e.
\end{equation}
By~\eqref{e:trapprop-2}, there exists $T>0$ such that
\begin{equation}
\label{e:chiprop}
\varphi_{t_1}(\supp\chi)\cap \varphi_{-t_2}(\supp\chi)\cap \supp(1-\chi)\cap S^*M=\emptyset\quad\text{for all }
t_1,t_2\geq T.
\end{equation}
This implies the following
\begin{lemma}
\label{l:stepFlow}
Let $\chi,T$ satisfy~\eqref{e:chiprop1}, \eqref{e:chiprop}. Then for all $t_0\geq T,t\geq 0$,
\begin{align}
  \label{e:stepFlow-1} 
\varphi_{t_0+T}(\supp\chi^+_t)\cap\supp(\chi-\chi^+_{t+t_0})\cap S^*M&=\emptyset,\\
  \label{e:stepFlow-2}
\varphi_{-t_0-T}(\supp\chi^-_t)\cap\supp(\chi-\chi^-_{t+t_0})\cap S^*M&=\emptyset,\\
  \label{e:stepFlow-3}
\varphi_{-t_0}(\supp\chi)\cap \supp(1-\chi)\cap \Gamma_+\cap S^*M&=\emptyset,\\
  \label{e:stepFlow-4}
\varphi_{t_0}(\supp\chi)\cap \supp(1-\chi)\cap \Gamma_-\cap S^*M&=\emptyset.
\end{align}
\end{lemma}
\begin{proof}
For~\eqref{e:stepFlow-1} it is enough to show that
$$
\varphi_{t+t_0+T}(\supp\chi)\cap\supp \chi\cap \varphi_{t+t_0}(\supp(1-\chi))\cap S^*M=\emptyset
$$
which follows immediately by applying $\varphi_{t+t_0}$ to~\eqref{e:chiprop} with $t_1=T,t_2=t+t_0$.

For~\eqref{e:stepFlow-2} it is enough to show that
$$
\varphi_{-t-t_0-T}(\supp\chi)\cap \supp\chi\cap  \varphi_{-t-t_0}(\supp(1-\chi))\cap S^*M=\emptyset
$$
which follows immediately by applying $\varphi_{-t-t_0}$ to~\eqref{e:chiprop} with $t_1=t+t_0,t_2=T$.

To show~\eqref{e:stepFlow-3}, choose $(x,\xi)$ in the left-hand side of this equation.
Since $(x,\xi)\in\Gamma_+$, by~\eqref{e:trapprop-1} we have $(x,\xi)\in\varphi_{t_1}(\supp\chi)$
for all $t_1\geq 0$ large enough depending on $(x,\xi)$.
Then
$$
(x,\xi)\in \varphi_{-t_0}(\supp\chi)\cap \varphi_{t_1}(\supp\chi)\cap\supp(1-\chi)\cap S^*M
$$
which is impossible by~\eqref{e:chiprop} with $t_2=t_0$, as soon as $t_1\geq T$.

Finally, to show~\eqref{e:stepFlow-4}, choose $(x,\xi)$ in the left-hand side of this equation. Since $(x,\xi)\in \Gamma_-$, by~\eqref{e:trapprop-1} we have $(x,\xi)\in \varphi_{-t_2}(\supp \chi)$ for all $t_2\geq 0$ large enough depending on $(x,\xi)$. Then 
$$
(x,\xi)\in \varphi_{t_0}(\supp \chi)\cap \varphi_{-t_2}(\supp \chi)\cap \supp (1-\chi)\cap S^*M
$$
which is impossible by~\eqref{e:chiprop} with $t_1=t_0$ as soon as $t_2\geq T$.
\end{proof}

\section{Proof of the Weyl upper bound}
  \label{s:weyl}

In this section, we prove Theorem~\ref{thm:fractalWeyl}, following
the method of~\cite{DyBo}. We use the function~$\chi$ and the constant $T$
satisfying~\eqref{e:chiprop1}, \eqref{e:chiprop}. We also assume that
$\chi$ is chosen to be homogeneous of degree 0 near~$S^*M$
and $\supp\chi\subset \{r<r_0\}\cap \{|\xi|_g\leq 2\}$.
We fix $h$-dependent
\begin{equation}
  \label{e:rho-fix}
\rho,\rho'\in [0,1),\quad
{1\over 2}\max(\rho,\rho')<\nu<{1\over 2},\quad
\rho t_e,\rho' t_e\geq C_0,
\end{equation}
with $C_0$ a large constant,
$\rho,\rho'$ chosen at the end of the proof,
and $\nu$ independent of $h$,
and define the following functions using~\eqref{e:t-e-h} and~\eqref{e:chi-pm}:
$$
\chi_+:=\chi^+_{\rho t_e},\quad
\chi_-:=\chi^-_{\rho' t_e},
$$
which both lie in $S^{\comp}_{h,\nu}(T^*M)$ by~\eqref{e:symbols-good}. We also use
a function
\begin{equation}
  \label{e:chi-E}
\chi_E\in \mathscr S(\mathbb R),\quad
\chi_E(0)=1,\quad
\supp\widehat\chi_E\subset (-1,1).
\end{equation}

\subsection{Approximate inverse}
  \label{s:weyl-inverse}

We first construct an approximate inverse for the complex scaled operator $P_\theta-\omega^2$
(see~\S\ref{s:resolvent-infinity}),
arguing similarly to the proof of~\cite[Proposition~2.1]{DyBo} and using the results of~\S\ref{s:cutoffs}.
See~\eqref{e:big-Omega} for the definitions of $\widetilde\Omega,\tilde\beta$.
\begin{lemma}
\label{l:ultimateInverse}
Fix $\varepsilon_0>0$. Then there exist $h$-dependent
families of operators holomorphic in $\omega\in\widetilde\Omega$
\begin{align}
  \label{e:ultimZ}
\mathcal Z(\omega):L^2(M)\to H^2(M),&\qquad
\|\mathcal Z(\omega)\|_{L^2(M)\to H^2_h(M)}\leq Ch^{-1}
e^{(\tilde\beta+\varepsilon_0)(\rho+\rho')t_e},\\
  \label{e:ultimJ}
\mathcal J(\omega):H^2(M)\to H^{2}(M),&\qquad
\|\mathcal J(\omega)\|_{H^{2}_h(M)\to H^{2}_h(M)}\leq C e^{(-h^{-1}\Im\omega+\varepsilon_0)\rho' t_e},
\end{align}
such that for all $\omega\in\widetilde\Omega$
and the constant $C_0$ in~\eqref{e:rho-fix} chosen large enough, we have on $H^2(M)$
\begin{equation}
  \label{e:ultima}
I=\mathcal Z(\omega)(P_\theta-\omega^2)+\mathcal J(\omega)\Op_h(\chi_-)\Op_h(\chi_+)\chi_E\Big({-h^2\Delta_g-\omega^2\over h}\Big)+\mathcal R(\omega),
\end{equation}
and the remainder $\mathcal R(\omega)$ is $\mathcal O(h^\infty)_{\Psi^{-\infty}}$.
\end{lemma}
\begin{proof}
Throughout the proof we will assume that $\omega\in\widetilde\Omega$;
the operators we construct are holomorphic in~$\omega$.
Fix $\varepsilon_1>0$ to be chosen at the end of the proof.
We first show that
\begin{gather}
  \label{e:monstra-1}
I=Z_-(\omega)(P_\theta-\omega^2)+J_-(\omega)\Op_h(\chi_-)+\mathcal O(h^\infty)_{\Psi^{-\infty}},\\
  \label{e:monstrz-1}
\|Z_-(\omega)\|_{H^s_h\to H^{s+2}_h}\leq C_{s,\varepsilon_1} h^{-1}\exp\big((1
+\varepsilon_1)\tilde\beta\rho' t_e\big),\\
  \label{e:monstrj-1}
\|J_-(\omega)\|_{H^{-N}_h\to H^N_h}\leq C_{N,\varepsilon_1}\exp\Big(
\Big(-{\Im\omega\over h}+\varepsilon_1\tilde\beta\Big)\rho't_e\Big).
\end{gather}
For that, fix $t_0$ bounded independently of $h$ and such that
$$
t_0>{2T\over \varepsilon_1},\quad
L:={\rho't_e\over t_0}\in\mathbb N.
$$
We apply Lemma~\ref{l:iterated-parametrix} to
$$
a_j=\chi^-_{t_0 j},\quad
t_1:=t_0+T.
$$
Indeed, we have $a_0=\chi^2=1$ in an $h$-independent neighborhood of $K\cap S^*M$
and $a_L=\chi_-$. To verify~\eqref{e:dyncon}, we first write by~\eqref{e:stepFlow-2}
with $t=t_0j$,
\begin{equation}
  \label{e:i-still-hate-topology}
\varphi_{-t_1}(\supp a_j)\cap \supp (\chi-a_{j+1})\cap S^*M=\emptyset.
\end{equation}
On the other hand, by~\eqref{e:stepFlow-3}
\begin{equation}
  \label{e:i-hate-topology}
\supp a_j\subset\supp \chi,\qquad
\varphi_{-t_1}(\supp\chi)\cap \supp(1-\chi)\cap \Gamma_+\cap S^*M=\emptyset.
\end{equation}
Since $\chi$ is independent
of $h$, $a_j,\chi$ are homogeneous of order 0 near $S^*M$, and
$$
\supp(1-a_{j+1})\ \subset\ \supp(1-\chi)\cup\supp(\chi-a_{j+1}),
$$
we see that $\varphi_{-t_1}(\supp a_j)\cap \supp(1-a_{j+1})$ is contained
in an $h$-independent compact set not intersecting $\Gamma_+\cap S^*M$ and
\eqref{e:dyncon} follows by making $V$ the complement of this compact set.
Finally, to satisfy~\eqref{e:dyncon-2}, we take $r_1$ large enough depending
on $t_0$. Now Lemma~\ref{l:iterated-parametrix} applies and
gives~\eqref{e:monstra-1}--\eqref{e:monstrj-1}.

We next show that
\begin{gather}
  \label{e:monstra-2}
\Op_h(\chi)=Z_+(\omega)(P_\theta-\omega^2)+\Op_h(\chi_+)+\mathcal O(h^\infty)_{\Psi^{-\infty}},\\
  \label{e:monstrz-2}
\|Z_+(\omega)\|_{H^s_h\to H^{s+2}_h}\leq C_{s,\varepsilon_1}h^{-1}\exp\big((1
+\varepsilon_1)\tilde\beta\rho t_e\big).
\end{gather}
For that, we fix $t_0$ bounded independently of $h$ and such that
$$
t_0>{2T\over \varepsilon_1},\quad
L:={\rho t_e\over t_0}-1\in\mathbb N.
$$
We apply Lemma~\ref{l:iterated-parametrix} to
$$
a_j=\chi-\chi^+_{t_0(L+1-j)},\quad
t_1:=t_0+T.
$$
Then $a_0=\chi-\chi_+$ and $a_L=\chi-\chi^+_{t_0}$. By~\eqref{e:stepFlow-3},
we have $\supp a_L\cap \Gamma_+\cap S^*M=\emptyset$; since
$a_L$ is independent of $h$, by Lemma~\ref{l:infinity-A1} we have for an appropriate choice of $Q$
\begin{equation}
  \label{e:addihelp-1}
\Op_h(a_L)=\Op_h(a_L)\mathcal R_Q(\omega)(P_\theta-\omega^2)+\mathcal O(h^\infty)_{\Psi^{-\infty}}.
\end{equation}
To verify~\eqref{e:dyncon}, \eqref{e:dyncon-2} we argue as in the proof of \eqref{e:monstra-1}--\eqref{e:monstrj-1}
above, using~\eqref{e:i-still-hate-topology} (which follows from~\eqref{e:stepFlow-1} with $t=t_0(L-j)$)
and \eqref{e:i-hate-topology}. Now Lemma~\ref{l:iterated-parametrix} applies and, combined with~\eqref{e:addihelp-1},
gives~\eqref{e:monstra-2}, \eqref{e:monstrz-2}.

We also have
\begin{equation}
\label{e:monstra-3}
  \Op_h(\chi_-)=Z_\chi(\omega)(P_\theta-\omega^2)+\Op_h(\chi_-)\Op_h(\chi)+\mathcal O(h^\infty)_{\Psi^{-\infty}},\quad
\|Z_\chi(\omega)\|_{H^s_h\to H^{s+2}_h}\leq C_sh^{-1}.
\end{equation}
Indeed, choose $C_0$ in~\eqref{e:rho-fix} large enough
so that $C_0\geq 2T$. Similarly to~\eqref{e:chiprop} we have
for some $\varepsilon'>0$
$$
\supp(\chi_-)\cap \supp(1-\chi)\cap \{1-\varepsilon'\leq |\xi|_g\leq 1+\varepsilon'\}\subset
\varphi_{-T}(\supp\chi)\cap \supp (1-\chi).
$$
The right-hand side is a compact set which by~\eqref{e:stepFlow-3}
does not intersect $\Gamma_+\cap S^*M$.
Now \eqref{e:monstra-3} follows by Lemma~\ref{l:infinity-A1}
applied to the operator $\Op_h(\chi_-)\Op_h(1-\chi)$.

Finally, put
$$
Z_E(\omega):=h^{-1}\psi_E\Big({-h^2\Delta_g-\omega^2\over h}\Big),\quad
\psi_E(\lambda)={1-\chi_E(\lambda)\over\lambda}.
$$
It follows from~\eqref{e:chi-E} that $\supp\widehat \psi_E\subset (-1,1)$, in particular $\psi_E$ is entire and
$Z_E$ can be defined.
Then
\begin{equation}
  \label{e:monstra-4.0}
I=Z_E(\omega)(-h^2\Delta_g-\omega^2)+\chi_E\Big({-h^2\Delta_g-\omega^2\over h}\Big),\quad
\|Z_E(\omega)\|_{H^{s}_h\to H^s_h}\leq C_sh^{-1}.
\end{equation}
By~\eqref{e:h-pseudolocality} and the fact that $P_\theta=-h^2\Delta_g$ on $\{r<r_1\}$,
we see that as long as $r_1>r_0+10$, we have
\begin{equation}
  \label{e:monstra-4}
\Op_h(\chi_+)=\Op_h(\chi_+)Z_E(\omega)(P_\theta-\omega^2)+\Op_h(\chi_+)\chi_E\Big({-h^2\Delta_g-\omega^2\over h}\Big)
+\mathcal O(h^\infty)_{\Psi^{-\infty}}.
\end{equation}
Combining~\eqref{e:monstra-1}, \eqref{e:monstra-3}, \eqref{e:monstra-2}, 
\eqref{e:monstra-4}, we obtain~\eqref{e:ultima} with
$$
\begin{aligned}
\mathcal Z(\omega)&=Z_-(\omega)+J_-(\omega)\Big(
Z_\chi(\omega)+\Op_h(\chi_-)\big(
Z_+(\omega)+\Op_h(\chi_+)Z_E(\omega)
\big)
\Big),\\
\mathcal J(\omega)&=J_-(\omega),
\end{aligned}
$$
and~\eqref{e:ultimZ}, \eqref{e:ultimJ} follow from \eqref{e:monstrz-1},
\eqref{e:monstrj-1}, \eqref{e:monstrz-2}, \eqref{e:monstra-3}, \eqref{e:monstra-4.0}
as long as we choose $\varepsilon_1<\varepsilon_0/\beta$.
\end{proof}

\subsection{Proof of Theorem~\texorpdfstring{\ref{thm:fractalWeyl}}{2}}

Fix $\varepsilon_0>0$ and let
$$
\mathcal A(\omega):=\mathcal J(\omega)\Op_h(\chi_-)\Op_h(\chi_+)\chi_E\Big({-h^2\Delta_g-\omega^2\over h}\Big)+\mathcal R(\omega)
$$
be the operator featured in Lemma~\ref{l:ultimateInverse}. Then
$\mathcal A(\omega)$ is a Hilbert--Schmidt operator on $H^2_h(M)$ and its
Hilbert--Schmidt norm is estimated by~\eqref{e:HS-estimate} and~\eqref{e:ultimJ}:
\begin{equation}
  \label{e:A-HS-norm}
\begin{aligned}
\|\mathcal A(\omega)\|^2_{\HS}&\leq \|\mathcal J(\omega)\|^2_{H^2_h\to H^2_h}
\cdot \Big\|\Op_h(\chi_-)\Op_h(\chi_+)\chi_E\Big({-h^2\Delta_g-\omega^2\over h}\Big)\Big\|^2_{\HS}
+\mathcal O(h^\infty)
\\&\leq
C h^{1-d}e^{2(-h^{-1}\Im\omega+\varepsilon_0)\rho' t_e}\cdot
\mu_L(S^*M\cap \supp\chi_+\cap\supp\chi_-)+\mathcal O(h^\infty)
\\&\leq
C h^{1-d}e^{2(-h^{-1}\Im\omega+\varepsilon_0)\rho' t_e}\cdot
\mathcal V\big((\rho+\rho')t_e\big)+\mathcal O(h^\infty)
=:V_{\rho,\rho',\varepsilon_0,h}(-h^{-1}\Im \omega)
\end{aligned}
\end{equation}
where we use~\eqref{e:V-t} and the fact that
$$
\supp\chi_+\cap \supp\chi_-
\ \subset\ \varphi_{\rho t_e}\big(\mathcal T((\rho+\rho')t_e)\big).
$$
Consider the Fredholm determinant
$$
F(\omega)=\det(I-\mathcal A(\omega)^2),\quad
\omega\in\widetilde\Omega.
$$
We have by~\eqref{e:A-HS-norm}
\begin{equation}
  \label{e:F-upper}
|F(\omega)|\leq \exp\big(\|\mathcal A(\omega)^2\|_{\tr}\big)
\leq \exp\big(\|\mathcal A(\omega)\|_{\HS}^2\big)
\leq \exp\big(V_{\rho,\rho',\varepsilon_0,h}(\tilde\beta)\big)\quad\text{for all }\omega\in\widetilde\Omega.
\end{equation}
On the other hand, if we put $\omega_0:=1+ih\in\widetilde\Omega$, then
by~\eqref{e:ultimJ} the norm $\|\mathcal A(\omega_0)\|_{H^2_h\to H^2_h}$
is bounded above by $1\over 2$ as long as
the constant $C_0$ in~\eqref{e:rho-fix} is large enough.
Therefore,
we have $\|(I-\mathcal A(\omega_0))^{-1}\|_{H^2_h\to H^2_h}\leq 2$ and thus
\begin{equation}
  \label{e:F-lower}
\begin{aligned}
|F(\omega_0)|^{-1}&=|\det(I-\mathcal A(\omega_0)^2)^{-1}|
=\big|\det \big(I+\mathcal A(\omega_0)^2(I-\mathcal A(\omega_0)^2)^{-1}\big)\big|\\
&\leq \exp\big(\|\mathcal A(\omega_0)^2(I-\mathcal A(\omega_0)^2)^{-1}\|_{\tr}\big)
\leq \exp\big(2\|\mathcal A(\omega_0)\|_{\HS}^2\big)
\leq \exp\big(2V_{\rho,\rho',\varepsilon_0,h}(\tilde\beta)\big).
\end{aligned}
\end{equation}
By~\eqref{e:ultima} we have
$$
(P_\theta-\omega^2)^{-1}=\big(I-\mathcal A(\omega)^2\big)^{-1}\big(I+\mathcal A(\omega)\big)\mathcal Z(\omega).
$$
Therefore, the poles of $(P_\theta-\omega^2)^{-1}$ in $\widetilde\Omega$
are contained in the set of poles of $(I-\mathcal A(\omega)^2)^{-1}$,
that is in the set of zeroes of $F(\omega)$, counting with multiplicity.
(The multiplicities are handled using Gohberg--Sigal theory,
see for example~\cite[\S C.4]{ZwScat}.) By~\eqref{e:F-upper},
\eqref{e:F-lower}, Jensen's bound on the number of zeroes of $F(\omega)$
(see for instance~\cite[Lemma~4.4]{DyatlovJin}; we dilate
the regions~\eqref{e:h-NRB-region}, \eqref{e:big-Omega}
by $h^{-1}$), and the relation
of the poles of $(P_\theta-\omega^2)^{-1}$ with resonances of $\Delta_g$,
we see that the bound
\begin{equation}
  \label{e:Weyl-basically-proved}
\mathcal N(R,\beta)\leq C R^{d-1}\exp\big(2(\tilde\beta+\varepsilon_0)\rho't_e(R)\big)
\cdot\mathcal V\big((\rho+\rho')t_e(R)\big)+\mathcal O(R^{-\infty})
\end{equation}
holds for all $\rho,\rho'\in [0,1)$ satisfying~\eqref{e:rho-fix}, $\varepsilon_0>0$, and $\tilde\beta>\beta$,
with $t_e(R)$ defined in~\eqref{e:t-e-R};
here the constant $C$ depends on $\tilde\beta$.
We assume that $K\cap S^*M\neq\emptyset$, since otherwise there is a resonance
free strip of arbitrarily large size (see for instance~\cite[Theorem~6.9]{ZwScat}).
Then by~\eqref{e:volume-nonzero}, we may remove the $\mathcal O(R^{-\infty})$ remainder in~\eqref{e:Weyl-basically-proved}.

Now, put $\rho':=C_0/t_e(R)$, where $C_0$ is the constant in~\eqref{e:rho-fix},
and $\rho:=1-\varepsilon_0$, $\tilde\beta:=\beta+\varepsilon_0$. Then~\eqref{e:Weyl-basically-proved} implies
(using~\eqref{e:volume-decreasing})
\begin{equation}
  \label{e:WBP-1}
  \mathcal N(R,\beta)\leq CR^{d-1}\cdot\mathcal V\big((1-\varepsilon_0)t_e(R)\big).
\end{equation}
If we instead put $\rho:=\rho':=1-2\beta^{-1}\varepsilon_0$, $\tilde\beta:=\beta+\varepsilon_0$,
then~\eqref{e:Weyl-basically-proved} implies
\begin{equation}
  \label{e:WBP-2}
  \mathcal N(R,\beta)\leq CR^{d-1}\exp\big(2\beta t_e(R)\big)
\cdot\mathcal V\big(2(1-2\beta^{-1}\varepsilon_0)t_e(R)\big).
\end{equation}
Choosing $\varepsilon_0$ small enough, we see that~\eqref{e:WBP-1}
and~\eqref{e:WBP-2} imply the bound~\eqref{e:fractalWeyl}, finishing
the proof of Theorem~\ref{thm:fractalWeyl}.

\section{Proof of wave decay on average}
  \label{s:wave}
  
\subsection{Hilbert--Schmidt bound}
  \label{s:wave-hs-bound}

We first use the results of~\S\ref{s:reduce-wave} to obtain a Hilbert--Schmidt
bound for the wave propagator. Assume that $\chi\in S^0_h(T^*M;[0,1])$
satisfies for some $r_2>r_0$ and $\varepsilon_E>0$,
$$
\supp\chi\subset \{r<r_2\},\quad
\supp(1-\chi)\cap \{|\xi|_g^2\in [1-\varepsilon_E,1+\varepsilon_E]\}\cap \{r\leq r_0\}=\emptyset.
$$
Put $T:= \sqrt{r_2^2-r_0^2}$. By~\eqref{e:convexinf-1} the following
stronger version of~\eqref{e:chiprop} holds:
\begin{equation}
  \label{e:extra-convex}
\varphi_{t_1}(\supp\chi)\cap \varphi_{-t_2}(\supp\chi)\cap \supp(1-\chi)\cap \{|\xi|_g^2\in [1-\varepsilon_E,1+\varepsilon_E]\}=\emptyset\quad\text{for all }t_1,t_2\geq T.
\end{equation}
Take an energy cutoff function $\psi_2\in C_c^\infty(\mathbb R)$ such that
\begin{equation}
\label{e:psi2}
\supp\psi_2\subset (1-\varepsilon_E,1+\varepsilon_E).
\end{equation}
Fix constants $0\leq \rho<2\nu<1$ and denote by $t_e$ the Ehrenfest time, see~\eqref{e:t-e-h}.
\begin{lemma}
\label{l:propLong}
Fix $\varepsilon_0\in (0,1)$. Then for each $t\in [5\varepsilon_0^{-1}T,\rho t_e]$,
\begin{equation}
  \label{e:propLong}
\|\Op_h(\chi^2)U(2t)\psi_2(-h^2\Delta_g)\Op_h(\chi^2)\|^2_{\HS}
\leq Ch^{-d}\mathcal V\big(2(1-\varepsilon_0)t\big)+\mathcal O(h^\infty).
\end{equation}
\end{lemma}
\begin{proof}
Fix $t_1$ bounded independently of $h$ and such that
$$
t_1\geq {T\over\varepsilon_0},\quad
L:={t\over t_1}\in\mathbb N,\quad
L\geq 5.
$$
Put $t_0:=t_1-T\geq 0$.
Fix $\psi_3\in C_c^\infty(\mathbb R;[0,1])$ such that
for some $\tilde\varepsilon_E<\varepsilon_E$
$$
\supp\psi_3\subset (1-\varepsilon_E,1+\varepsilon_E),\quad
\supp(1-\psi_3)\cap [1-\tilde\varepsilon_E,1+\tilde\varepsilon_E]=\emptyset,\quad
\supp\psi_2\subset (1-\tilde\varepsilon_E,1+\tilde\varepsilon_E).
$$
Put
$$
\widetilde\chi:=\psi_3(|\xi|_g^2)\chi,\quad
\widetilde\chi^\pm_s:=\widetilde\chi(\chi\circ\varphi_{\mp s}).
$$
Similarly to~\eqref{e:symbols-good},
$\widetilde\chi^\pm_s\in S^{\comp}_{h,\nu}(M)$
for $|s|\leq \rho t_e$.
Using~\eqref{e:extra-convex}, the proof of~\eqref{e:stepFlow-1}, \eqref{e:stepFlow-2}
gives for all $s\geq 0$
\begin{align}
  \label{e:stepFlowBetter-1}
\varphi_{t_1}(\supp\widetilde\chi^+_s)\cap\supp(\widetilde\chi-\widetilde\chi^+_{s+t_0})&=\emptyset,
\\
  \label{e:stepFlowBetter-2}
\varphi_{-t_1}(\supp\widetilde\chi^-_s)\cap\supp(\widetilde\chi-\widetilde\chi^-_{s+t_0})&=\emptyset.
\end{align}
We have
$\psi_2(-h^2\Delta_g)\Op_h(\chi^2-\widetilde\chi_0^+)=\mathcal O(h^\infty)_{\Psi^{-\infty}}$.
Moreover, since $\psi_2(-h^2\Delta_g)$ commutes with $U(2t)$
$$
\Op_h(\chi^2-\widetilde\chi_0^-)U(2t)\psi_2(-h^2\Delta_g)\Op_h(\widetilde\chi_0^+)=\mathcal O(h^\infty)_{\Psi^{-\infty}}.
$$
It follows that
\begin{equation}
  \label{e:beaver-1}
\Op_h(\chi^2)U(2t)\psi_2(-h^2\Delta_g)\Op_h(\chi^2)
=\Op_h(\widetilde\chi_0^-)U(2t)\psi_2(-h^2\Delta_g)\Op_h(\widetilde\chi_0^+)
+\mathcal O(h^\infty)_{\Psi^{-\infty}}.
\end{equation}
From~\eqref{e:beaver-1} and Lemma~\ref{l:super-sandwich}
(taking $\tilde\varepsilon_E$ in place of $\varepsilon_E$) we get
\begin{equation}
  \label{e:beaver-2}
\begin{gathered}
\Op_h(\chi^2)U(2t)\psi_2(-h^2\Delta_g)\Op_h(\chi^2)\\
=
\Op_h(\widetilde\chi_0^-)U(t_1)\big(\Op_h(\widetilde\chi)U(t_1)\big)^{2L-1}
\psi_2(-h^2\Delta_g)\Op_h(\widetilde\chi_0^+)
+\mathcal O(h^\infty)_{\Psi^{-\infty}}.
\end{gathered}
\end{equation}
We next transform the right-hand side of~\eqref{e:beaver-2} into an expression
involving the cutoffs $\widetilde\chi_{t}^\pm$. First of all, we claim that
\begin{equation}
  \label{e:badger-1}
\Big(\big(\Op_h(\widetilde\chi)U(t_1)\big)^L
-\Op_h(\widetilde\chi_{Lt_0}^+)U(t_1)
\cdots\Op_h(\widetilde\chi_{t_0}^+)
U(t_1)\Big)\psi_2(-h^2\Delta_g)\Op_h(\widetilde\chi_0^+)
=\mathcal O(h^\infty)_{\Psi^{-\infty}}.
\end{equation}
Indeed, the left-hand side of~\eqref{e:badger-1}
is equal to $\sum_{\ell=1}^L B^+_\ell$ where
$$
B^+_\ell:=\big(\Op_h(\widetilde\chi)U(t_1)\big)^{L-\ell}
\Op_h(\widetilde\chi-\widetilde\chi_{\ell t_0}^+)U(t_1)
\Op_h(\widetilde\chi_{(\ell-1)t_0}^+)U(t_1)
\cdots \Op_h(\widetilde\chi_{t_0}^+)U(t_1)\psi_2(-h^2\Delta_g)\Op_h(\widetilde\chi_0^+),
$$
in particular
$B^+_1=\big(\Op_h(\widetilde\chi)U(t_1)\big)^{L-1}\Op_h(\widetilde\chi-\widetilde\chi_{t_0}^+)U(t_1)\psi_2(-h^2\Delta_g)\Op_h(\widetilde\chi_0^+)$.
By Lemma~\ref{l:egorov} and~\eqref{e:stepFlowBetter-1}
with $s:=(\ell-1)t_0$ we have
$$
\Op_h(\widetilde\chi-\widetilde\chi_{\ell t_0}^+)U(t_1)
\Op_h(\widetilde\chi_{(\ell-1)t_0}^+)U(-t_1)=\mathcal O(h^\infty)_{\Psi^{-\infty}}
$$
for $\ell=2,\dots,L$ and a similar argument with $s:=0$ gives
$$
\Op_h(\widetilde\chi-\widetilde\chi_{t_0}^+)U(t_1)
\psi_2(-h^2\Delta_g)\Op_h(\widetilde\chi_{0}^+)U(-t_1)=\mathcal O(h^\infty)_{\Psi^{-\infty}}.
$$
Therefore $B^+_\ell=\mathcal O(h^\infty)_{\Psi^{-\infty}}$ and~\eqref{e:badger-1} follows.

We next claim that
\begin{equation}
  \label{e:badger-2}
\begin{aligned}
\Op_h(\widetilde\chi_0^-)U(t_1)\big(\Op_h(\widetilde\chi)U(t_1)\big)^{L-1}
\Op_h(\widetilde\chi_{Lt_0}^+)&\\
-\Op_h(\widetilde\chi_0^-)U(t_1)\cdots\Op_h(\widetilde\chi_{(L-1)t_0}^-)
U(t_1)\Op_h\big(\widetilde\chi_{Lt_0}^-(\chi\circ\varphi_{-Lt_0})\big)
&=\mathcal O(h^\infty)_{\Psi^{-\infty}}.
\end{aligned}
\end{equation}
Indeed, the left-hand side of~\eqref{e:badger-2} has the form
$\sum_{\ell=1}^L B^-_\ell$ where
$$
B^-_\ell:=\Op_h(\widetilde\chi^-_0)U(t_1)\cdots
\Op_h(\widetilde\chi^-_{(\ell-1)t_0})U(t_1)
\Op_h(\widetilde\chi-\widetilde\chi^-_{\ell t_0})
U(t_1)\big(\Op_h(\widetilde\chi)U(t_1)\big)^{L-\ell-1}
\Op_h(\widetilde\chi_{Lt_0}^+)
$$
for $\ell=1,\dots,L-1$ and
$$
B^-_L:=\Op_h(\widetilde\chi^-_0)U(t_1)\cdots \Op_h(\widetilde\chi^-_{(L-1)t_0})
U(t_1)\Op_h\big((\widetilde\chi-\widetilde\chi^-_{Lt_0})(\chi\circ\varphi_{-Lt_0})\big).
$$
By Lemma~\ref{l:egorov} and~\eqref{e:stepFlowBetter-2} with $s:=(\ell-1)t_0$,
$\ell=1,\dots,L-1$, we have
$$
U(-t_1)\Op_h(\widetilde\chi^-_{(\ell-1)t_0})U(t_1)
\Op_h(\widetilde\chi-\widetilde\chi^-_{\ell t_0})=\mathcal O(h^\infty)_{\Psi^{-\infty}}
$$
and a similar argument with $s:=(L-1)t_0$ gives
$$
U(-t_1)\Op_h(\widetilde\chi^-_{(L-1)t_0})
U(t_1)\Op_h\big((\widetilde\chi-\widetilde\chi^-_{Lt_0})(\chi\circ\varphi_{-Lt_0}))=\mathcal O(h^\infty)_{\Psi^{-\infty}}.
$$
Therefore $B_\ell^-=\mathcal O(h^\infty)_{\Psi^{-\infty}}$ and~\eqref{e:badger-2} follows.

Combining~\eqref{e:beaver-2}--\eqref{e:badger-2}, we obtain
$$
\begin{gathered}
\Op_h(\chi^2)U(2t)\psi_2(-h^2\Delta_g)\Op_h(\chi^2)
=
A_- A A_++\mathcal O(h^\infty)_{\Psi^{-\infty}},\\
A_-:=
\Op_h(\widetilde\chi_0^-)U(t_1)
\cdots
\Op_h(\widetilde\chi_{(L-1)t_0}^-)U(t_1),\\
A:=
\Op_h\big(\widetilde\chi^-_{Lt_0}(\chi\circ\varphi_{-Lt_0})\big),\\
A_+:=U(t_1)\Op_h(\widetilde\chi_{(L-1)t_0}^+)U(t_1)\cdots
\Op_h(\widetilde\chi_{t_0}^+)
U(t_1)\psi_2(-h^2\Delta_g)\Op_h(\widetilde\chi_0^+).
\end{gathered}
$$
In fact the remainder is $\mathcal O(h^\infty)_{\HS}$ since
its range consists of functions supported
in $\{r<r_2\}$.
By~\eqref{e:opernorm} and since $0\leq \widetilde\chi^\pm_s\leq 1$,
we have as $h\to 0$
$$
\|A_\pm\|_{L^2\to L^2}=\mathcal O(1).
$$
Therefore
\begin{equation}
  \label{e:beaver-hs-1}
\|\Op_h(\chi^2)U(2t)\psi_2(-h^2\Delta_g)\Op_h(\chi^2)\|_{\HS}
\leq C\|A\|_{\HS}+\mathcal O(h^\infty).
\end{equation}
Finally, we have by~\eqref{e:hsnorm}
$$
\|A\|_{\HS}^2\leq Ch^{-d}\Vol\big(\supp\widetilde\chi\cap \varphi_{Lt_0}(\supp\chi)
\cap\varphi_{-Lt_0}(\supp\chi)\big)
\leq Ch^{-d}\mathcal V(2Lt_0)\leq
Ch^{-d}\mathcal V\big(2(1-\varepsilon_0)t\big)
$$
where in the last inequality we use~\eqref{e:volume-decreasing}.
Combined with~\eqref{e:beaver-hs-1} this gives~\eqref{e:propLong}.
\end{proof}

\subsection{Concentration of measures}

Let $\mathcal E_R\subset L^2(\mathcal B)$ be as in the introduction, in particular
for some constant $c>0$
$$
N_R:=\dim \mathcal E_R=cR^d+o(R^{d}).
$$
Denote by $\mathcal S_R$ the unit sphere in $\mathcal E_R$.
Let $u_R\in\mathcal S_R$ be chosen randomly with respect to the
standard measure on the sphere.
\begin{lemma}
\label{l:probability}
Let $A:\mathcal E_R\to L^2(M)$ be a bounded linear operator
and take $R$ large enough so that $N_R\geq 10$.
Then for all $m\geq 10$,
\begin{equation}
  \label{e:probability}
\mathbb{P}\big(\|Au_R\|_{L^2(M)}> mN_R^{-1/2}\|A\|_{\HS}\big)\leq 2e^{-m^2/16}.
\end{equation}
\end{lemma}
\begin{proof}
Denote by $\mu$ the standard probability measure on $\mathcal S_R$
and let $e_1,\dots,e_{N_R}$ be an orthonormal basis of $\mathcal E_R$.
Consider the function $f(u)=\|Au\|_{L^2(M)}$, $u\in \mathcal S_R$.
We have
\begin{align*}
\mathbb{E}(f(u_R)^2)
&=\int_{\mathcal S_R} \big\la Au_R(a),Au_R(a)\big\ra_{L^2}\, d\mu(a)\\
&=\int_{\mathcal S_R}\sum_{k,j=1}^{N_R}\big\la a_kAe_k,a_jAe_j\big\ra_{L^2}\, d\mu(a)\\
&=\frac{1}{N_R}\sum_{k=1}^{N_R} \|Ae_k\|^2=\frac{1}{N_R}\|A\|_{\HS}^2.
\end{align*}
The function $f$ is Lipschitz continuous; indeed, for $u, v\in \mc{S}_R$ 
$$
\big|\|Au\|_{L^2}-\|Av\|_{L^2}\big|\leq \|A(u-v)\|_{L^2}\leq \|A\|_{\mathcal E_R\to L^2}\cdot\|u-v\|_{\mathcal E_R}
\leq
\|A\|_{\HS}\cdot\|u-v\|_{\mathcal E_R}.
$$
By the Levy concentration of measure theorem \cite[(2.6)]{Led}
\begin{equation}
  \label{e:levy}
\mathbb{P}\big(|f(u_R)-\mc{M}(f)|> \|A\|_{\HS}\cdot\eta\big)\leq 2e^{-(N_R-2)\eta^2/2}
\leq 2e^{-N_R\eta^2/4}
\quad\text{for all }\eta>0
\end{equation}
where $\mc{M}(f)$ is the \emph{median} of $f(u_R)$, namely the unique number with the properties
$$
\mathbb{P}\big( f(u_R)\geq \mc{M}(f)\big)\geq \frac{1}{2},\qquad \mathbb{P}\big(f(u_R)\leq \mc{M}(f)\big)\geq \frac{1}{2}.
$$
We next estimate the difference between $\mc{M}(f)$ and $\mathbb{E}(f(u_R))$. 
By~\eqref{e:levy}
\begin{align*} 
|\mathbb{E}(f(u_R))-\mc{M}(f)|\leq \mathbb E|f(u_R)-\mathcal M(f)|
&=\int_0^\infty \mathbb P\big(|f(u_R)-\mathcal M(f)|> r\big)\,dr
\\
&\leq 2\int_0^\infty\exp\Big(-{N_Rr^2\over 4\|A\|_{\HS}^2}\Big)\,dr
\\
&\leq 4\|A\|_{\HS}N_R^{-1/2}.
\end{align*}
Since $|\mathbb{E}(f(u_R))|\leq \sqrt{\mathbb{E}(f(u_R)^2)}$ by Jensen's inequality, we have
$$
\mathcal M(f)\leq 5\|A\|_{\HS}N_R^{-1/2}.
$$
Using~\eqref{e:levy} with $\eta:=(m-5)N_R^{-1/2}\geq {1\over 2}mN_R^{-1/2}$, we obtain for $m\geq 10$
$$
\mathbb P\big(f(u_R)>mN_R^{-1/2}\|A\|_{\HS}\big)\leq
\mathbb P\big(|f(u_R)-\mathcal M(f)|> \eta\|A\|_{\HS}\big)
\leq 2e^{-m^2/16}
$$
finishing the proof.
\end{proof}

\subsection{Proof of Theorem~\texorpdfstring{\ref{thm:decayOnAverage}}{3}}
  \label{s:proof-of-decay}

Recall from~\eqref{e:the-set-B} that $\mathcal B=\{r\leq r_1\}$ for some $r_1>r_0$.
With $\varepsilon'>0$ the parameter from~\eqref{e:E-R}, fix
$\varepsilon_E>0$ such that
\begin{equation}
  \label{e:game-of-epsilons}
[(1-\varepsilon')^2,(1+\varepsilon')^2]\subset
(1-\varepsilon_E,1+\varepsilon_E)
\end{equation}
and fix $\psi_2\in C_c^\infty(\mathbb R)$ such that
\begin{equation}
  \label{e:psi-2-finalized}
\supp\psi_2\subset (1-\varepsilon_E,1+\varepsilon_E),\quad
\supp(1-\psi_2)\cap [(1-\varepsilon')^2,(1+\varepsilon')^2]=\emptyset.
\end{equation}
Let $\psi\in C_c^\infty(\mathcal B^\circ)$ be chosen in Theorem~\ref{thm:decayOnAverage}.
Without loss of generality we assume that $|\psi|\leq 1$.
We assume that $R$ is large and put
$$
h:=R^{-1}.
$$
We use
the definition~\eqref{e:E-R} of the space $\mathcal E_R$ to show
the following microlocalization statement:
\begin{lemma}
  \label{l:basic-localization}
We have for all $u\in \mathcal E_R$
\begin{equation}
  \label{e:basic-localization}
\|(I-\psi_2(-h^2\Delta_g))\psi u\|_{L^2}=\mathcal O(h^\infty)\|u\|_{L^2}.
\end{equation}
\end{lemma}
\begin{proof}
Let $\{e_k\}$ be an orthonormal basis of $L^2(\mathcal B)$
with $(-\Delta_{\mathcal B}-\lambda_k^2)e_k=0$. Then it suffices
to show that for each $k$ such that $h\lambda_k\in [1-\varepsilon',1+\varepsilon']$,
we have
$$
\|(I-\psi_2(-h^2\Delta_g))\psi e_k\|_{L^2}=\mathcal O(h^\infty).
$$
Let $\psi'\in C_c^\infty(\mathcal B^\circ)$ satisfy
$\supp\psi\cap\supp(1-\psi')=\emptyset$.
Then $(1-\psi')(I-\psi_2(-h^2\Delta_g))\psi=\mathcal O(h^\infty)_{\Psi^{-\infty}}$,
therefore it suffices to show that
\begin{equation}
  \label{e:basic-localization-2}
\|B e_k\|_{L^2}=\mathcal O(h^\infty),\quad
B:=\psi'(I-\psi_2(-h^2\Delta_g))\psi\in\Psi^0_h(M).
\end{equation}
The Schwartz kernel of $B$ is compactly supported in $\mathcal B^\circ$.
The function $e_k$ solves the equation
$$
(-h^2\Delta_g-(h\lambda_k)^2)e_k=0\quad\text{in }\mathcal B^\circ
$$
and the operator $-h^2\Delta_g-(h\lambda_k)^2\in\Psi^2_h(\mathcal B^\circ)$
is elliptic on $\WFh(B)$ due to~\eqref{e:psi-2-finalized}.
Then~\eqref{e:basic-localization-2} follows from the semiclassical
elliptic estimate, see for instance~\cite[Theorem~E.32]{ZwScat}.
\end{proof}
Let $\chi'\in C_c^\infty(M)$ satisfy~\eqref{e:mes-2}
and fix $r_2>r_1$ such that $\supp\chi'\subset \{r<r_2\}$.
By Lemma~\ref{l:mega-sandwich} combined with~\eqref{e:basic-localization}
we have for all $u\in\mathcal S_R$
\begin{equation}
  \label{e:octopus}
\begin{aligned}
\|\psi U(s+t)\psi u\|_{L^2}
&\leq \|\psi U(s+t)\psi_2(-h^2\Delta_g)\psi u\|_{L^2}+\mathcal O(h^\infty)
\\
&\leq \|\chi'U(t)\psi_2(-h^2\Delta_g)\psi u\|_{L^2}+\mathcal O(h^\infty)
\end{aligned}
\end{equation}
for all $t\in [T_0,C_0h^{-1}]$, $s\in [0,C_0h^{-1}]$, where $T_0:=\sqrt{r_1^2-r_0^2}$.

Using~\eqref{e:octopus} and Lemmas~\ref{l:propLong}--\ref{l:probability}, we now give
\begin{proof}[Proof of Theorem \ref{thm:decayOnAverage}]
With $\varepsilon,\alpha>0$ the parameters
in the statement of Theorem~\ref{thm:decayOnAverage},
take $\varepsilon_0,\rho,\nu$ such that
$$
0<\varepsilon_0<\min\Big(\frac{\e}{4},\alpha,{1\over 10\Lambda_{\max}},{1\over 10}\Big),\quad
{1\over 1+\varepsilon_0}<\rho<2\nu<1.
$$
Let $t_e(R)$ be defined in~\eqref{e:t-e-R}.
Fix a sequence of times
$$
\e_0\log R=t_0<t_1<\cdots <t_L=2\rho t_e(R),\qquad t_i\leq (1+\e_0) t_{i-1},\quad i\geq 1
$$
with the following bound on $L$ (seen by rewriting the inequality above as $\log t_i\leq\log(1+\e_0)+\log t_{i-1}$)
$$
1\leq L\leq 1-{\log(\varepsilon_0\Lambda_{\max})\over\log (1+\varepsilon_0)}.
$$ 
Fix $\chi=\chi(x)\in C_c^\infty(M;[0,1])$ such that
\begin{equation}
  \label{e:snow}
\supp \chi\subset \{r<r_2\},\quad
\supp(1-\chi)\cap \{r\leq r_1\}=\emptyset,\quad
\supp(1-\chi)\cap \supp \chi'=\emptyset.
\end{equation}
We view $\chi$ as a function of $(x,\xi)\in T^*M$
and note that $\chi,\psi_2$ satisfy the assumptions of~\S\ref{s:wave-hs-bound}. Then
Lemma~\ref{l:propLong} (with $t:=t_i/2$) gives for all $i=1,\dots, L$
$$
\|\chi^2 U(t_i) \psi_2(-h^2\Delta_g)\chi^2 \|^2_{\HS} \leq Ch^{-d}\mathcal V\big((1-\varepsilon_0)t_i\big)
$$
where we remove the $\mathcal O(h^\infty)$ remainder by~\eqref{e:volume-nonzero}
using the assumption $K\neq\emptyset$.
Furthermore, $\chi^2\chi'=\chi'$ and $\chi^2\psi=\psi$, so 
\begin{equation} 
\label{e:pine}
\|\chi ' U(t_i)\psi_2(-h^2\Delta_g)\psi \|_{\HS}\leq Ch^{-d/2}\sqrt{\mathcal{V}((1-\e_0)t_i)}.
\end{equation}
Write $t_{L+1}:=C_0R$. Suppose that $t\in [\e_0\log R,C_0R]$. Then there exists $i\geq 0$ so that  $t\in [t_i,t_{i+1}]$.
By~\eqref{e:octopus} with $(t_i,t-t_i)$ taking the role of $(t,s)$
\begin{multline}
  \label{e:gorilla}
\mathbb P\Big[ \|\psi U(t)\psi u_R\|_{L^2}\leq m\sqrt{ \mathcal{V} \big((1-2\varepsilon_0)\min(t,2t_e(R))\big)}\text{ for all }t\in [t_i,t_{i+1}]\Big]\geq 
\\
\mathbb P\Big[ \|\chi' U(t_i)\psi_2(-h^2\Delta_g)\psi u_R\|_{L^2}\leq \frac{m}{2} \sqrt{\mathcal{V} \big((1-2\varepsilon_0)\min(t_{i+1},2t_e(R))\big)}\,\Big]
\end{multline}
where we again use~\eqref{e:volume-nonzero} and the monotonicity~\eqref{e:volume-decreasing} of $\mathcal{V}(t)$
to remove the $\mathcal O(h^\infty)$ error. 
Now, since $t_{i+1}\leq (1+\varepsilon_0) t_{i}$ for $i=0,\dots,L-1$ and $2t_e(R)\leq (1+\e_0)t_L$,
$$
(1-2\e_0)\min(t_{i+1},2t_e(R))\leq 
(1-2\e_0)(1+\e_0)t_i\leq 
(1-\e_0)t_i.
$$
Using \eqref{e:pine} and the monotonicity of $\mathcal{V}(t)$, we have
$$
N_R^{-1/2}\|\chi' U(t_i)\psi_2(-h^2\Delta_g) \psi \|_{\HS}\leq C\sqrt{\mathcal{V}((1-\varepsilon_0)t_i)}\leq C\sqrt{\mathcal{V}\big((1-2\varepsilon_0)\min(t_{i+1},2t_e(R))\big)}.
$$
Lemma~\ref{l:probability} applied to $A:=\chi'U(t_i)\psi_2(-h^2\Delta_g)\psi$ then implies
that there exists $C>0$ such that for all $m\geq C$
$$
 \mathbb P\Big[ \|\chi' U(t_i)\psi_2(-h^2\Delta)\psi u_R\|_{L^2}> \frac{m}{2}\sqrt{ \mathcal{V} \big((1-2\varepsilon_0)\min(t_{i+1},2t_e(R))\big)} \Big]\leq 2e^{-m^2/C}.
$$
Therefore, by~\eqref{e:gorilla}
$$
 \mathbb P\Big[ \|\psi U(t)\psi u_R\|_{L^2}\leq m\sqrt{ \mathcal{V} \big((1-2\varepsilon_0)\min(t,2t_e(R))\big)}\text{ for all }t\in [t_i,t_{i+1}] \Big]\geq 1- 2e^{-m^2/C}.
 $$
Taking an intersection of these events for $i=0,\dots, L$ then gives
$$
 \mathbb P\Big[ \|\psi U(t)\psi u_R\|_{L^2}\leq m\sqrt{ \mathcal{V} \big((1-2\varepsilon_0)\min(t,2t_e(R))\big)}\text{ for all }t\in [\e_0\log R,C_0R] \Big]\geq 1- 4Le^{-m^2/C},
$$
finishing the proof.
\end{proof}

\section{Examples}
\label{s:examples}

\subsection{Manifolds of revolution}

Consider the warped product $M=\re_r \times\mathbb S^{d-1}_\theta$ with metric
$$
g=dr^2+\alpha(r)^2g_0(\theta,d\theta)
$$
where $g_{0}$ is the round metric on the sphere, $\alpha\in C^\infty(\re;\re_+)$, and there exists $C>0$ so that
$$
\alpha(r)=| r|,\quad  |r|>C.
$$  
Then $M$ is a manifold with two Euclidean ends so Theorems~\ref{thm:fractalWeyl} and~\ref{thm:decayOnAverage} apply.
The symbol of the Laplacian is given
$$
p^2=\rho^2+\alpha^{-2}(r)p_0,\quad
p_0:=|\eta|_{g_0(\theta)}^2
$$
where $\rho,\eta$ denote the momenta dual to $r,\theta$. We compute
$$
2pH_p=H_{p^2}=2\rho\partial_r +2\alpha^{-3}(r)\alpha'(r)p_0\partial_{\rho}+\alpha^{-2}(r)H_{p_0}.
$$
Therefore, for a geodesic $(r(t),\theta(t),\rho(t),\eta(t))$, 
$$
\begin{cases}
\dot{r}=p^{-1}\rho\\
\dot{\rho}=p^{-1}\alpha^{-3}(r)\alpha'(r)p_0\\
\dot p_0=0.
\end{cases}
$$
Throughout this section, we assume that
\begin{equation}
\label{e:assume}
\begin{gathered}
\pm \alpha'(r)\geq 0\quad\text{ for }\pm r\geq 0.
\end{gathered}
\end{equation}
Notice that 
\begin{equation}
\label{e:ddot}
\ddot{r}=p^{-2}\alpha^{-3}(r)\alpha'(r)p_0.
\end{equation}
To understand trapping on $M$, we use
\begin{lemma}
\label{l:dynamicsRevolution}
For any geodesic $(r(t),\theta(t),\rho(t),\eta(t))\in \{p=1\}$, we have
for all $t\geq 0$
\begin{align}
  \label{e:DR-1}
\rho(0)r(0)\geq 0&\quad \Longrightarrow\quad |r(t)|\geq |r(0)|+|\rho(0)t|,\\
\rho(0)r(0)\leq 0&\quad \Longrightarrow\quad |r(-t)|\geq |r(0)|+|\rho(0)t|.
\end{align}
\end{lemma}
\begin{proof}
We prove~\eqref{e:DR-1} under the assumption $r(0)\geq 0$, $\rho(0)\geq 0$, with the other cases handled similarly.
By~\eqref{e:assume} and~\eqref{e:ddot}, we have $r(t)\ddot r(t)\geq 0$ for all $t$.
Moreover, $\dot r(0)\geq 0$. This implies that $r(t)\geq 0$ for all $t\geq 0$
and thus $\dot r(t)\geq \dot r(0)=\rho(0)$ for $t\geq 0$. This immediately gives~\eqref{e:DR-1}.
\end{proof}
Denote by $K\subset T^*M\setminus 0$ the trapped set, see~\eqref{e:gamma-def}.
Lemma \ref{l:dynamicsRevolution} implies that
$$
K\subset\{\alpha'(r)=0,\,\rho=0\}.
$$
On the other hand, if $\rho(0)=0$ and $\alpha'(r(0))=0$, then $r\equiv r(0)$ and hence 
\begin{equation}
\label{e:trappedRevolution}
K=\{\alpha'(r)=0,\,\rho=0\}.
\end{equation}

\subsection{Example with cylindrical trapping}
  \label{s:example-1}

We now consider two special examples of manifolds of revolution. First, let $M$ be given as above with (see Figure~\ref{f:examples})
$$
\alpha(r)=\begin{cases}1,&|r|\leq 2;\\
| r|,&| r|\geq 4.
\end{cases}
$$
such that $r\alpha'(r)>0$ when $|r|>2$. Then by \eqref{e:trappedRevolution}, 
$$
K=\{|r|\leq 2,\ \rho=0\}.
$$
We estimate $\mc{V}(t)$ when $t\gg 1$. 
Fix
$$
\mc{B}:=\{|r|\leq 3\}.
$$ 
Since $\dot \rho=0$ for $|r|\leq 2$, we have
$$
\{|r|\leq 1,\ |\rho|\leq p/t\}\subset \mc{T}_{\mc{B}}(t).
$$
On the other hand, suppose that $|\rho(0)|\geq 4p/t$. Then by Lemma~\ref{l:dynamicsRevolution}, 
$$
\max(|r(t)|,|r(-t)|)\geq 4.
$$
Therefore, 
$$
\varphi_{-t}(\mc{T}_{\mc{B}}(2t))\subset \{|r|\leq 3,\ |\rho|\leq 4p/t\}.
$$
In particular, this shows that there exists $C>0$ so that 
$$
C^{-1}t^{-1}\leq \mc{V}(t)\leq Ct^{-1}.
$$

\begin{figure}
\includegraphics{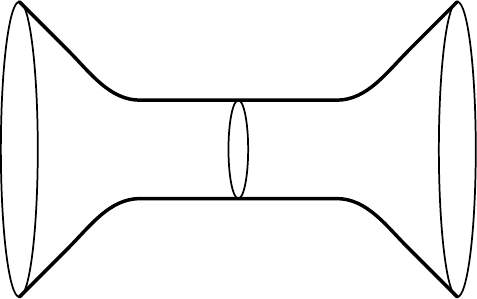}
\qquad\qquad\qquad
\includegraphics{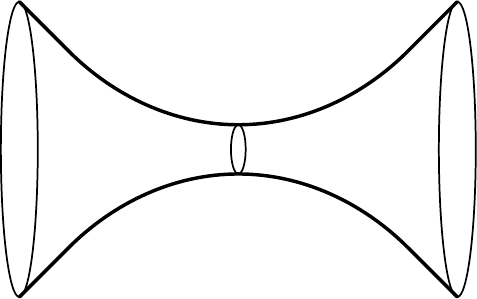}
\caption{Examples of surfaces of revolution studied in~\S\ref{s:example-1} (left)
and~\S\ref{s:example-2} (right).}
\label{f:examples}
\end{figure}

\subsection{Example with degenerate hyperbolic trapping}
  \label{s:example-2}

Next, we study a less degenerate situation. Fix an integer $ n\geq 2$ and let $M$ be given as above with (see Figure~\ref{f:examples})
$$
\alpha(r)=\begin{cases} 1+\displaystyle{r^{2n}\over 2}+\O{}(r^{2n+1}),&|r|\leq 1;\\
|r|,& |r| \geq 4.
\end{cases}
$$
such that $r \alpha'(r)>0$ for $r\neq 0$. Then by \eqref{e:trappedRevolution}
$$
K=\{r=0,\ \rho=0\}.
$$
Fix small  $\tau>0$ to be chosen later and let  
$$
\mc{B}=\{|r|\leq \tau\}.
$$
We consider the flow on $\{p=1\}=S^*M$, so that 
$$
\rho^2=1-\alpha(r)^{-2}p_0\geq 1-p_0.
$$
Recall that $p_0$ is constant on each geodesic.

We henceforth assume that $t\geq 1$.
Observe that if $p_0<1-\tau$,
then $|\rho(0)|>\tau^{1/2}$ and hence 
by Lemma~\ref{l:dynamicsRevolution}
$\max(|r(t)|,|r(-t)|)> \sqrt{\tau}\geq\tau$. Therefore
$$
\varphi_{-t}(\mc{T}_{\mc{B}}(2t))\cap S^*M\subset \{|r|\leq \tau,\ p_0 \geq 1-\tau\}.
$$
By symmetry considerations, to understand the set
$\varphi_{-t}(\mathcal T_{\mathcal B}(2t))\cap S^*M$
it suffices to consider the set of trajectories which satisfy
\begin{equation}
  \label{e:island}
p=1,\quad
p_0\geq 1-\tau,\quad
r(0)\geq 0,\quad
\rho(0)\geq 0,\quad
r(t)\leq \tau.
\end{equation}
\begin{lemma}
  \label{l:stuff}
Under the assumption~\eqref{e:island}, for $\tau>0$ fixed small enough
and large $t$ we have
\begin{gather}
  \label{e:stuff-1}
r(0)\leq Ct^{-{1\over n-1}},\\
  \label{e:stuff-2}
\rho(0)\leq Ct^{-{n\over n-1}}.
\end{gather}
\end{lemma}
\begin{proof}
Note that $0\leq r(0)\leq r(t)\leq\tau$.
Moreover, we have
$\alpha(r(0))\geq \sqrt{p_0}$.
Since $\dot r=\rho=\sqrt{1-\alpha(r)^{-2}p_0}$, we have
$$
t=\int_{r(0)}^{r(t)}\frac{dr}{ \sqrt{1-\alpha(r)^{-2}p_0}}
\leq \int_{r(0)}^\tau \frac{dr}{ \sqrt{1-\alpha(r)^{-2}p_0}}.
$$
Using the inequality $\alpha(r)-\alpha(s)\geq C^{-1}(r-s)r^{2n-1}$,
$0\leq s\leq r\leq \tau$, we have
$$
\begin{aligned}
t&\leq\int_{r(0)}^\tau \Big(1-{\alpha(r(0))^2\over \alpha(r)^2}\Big)^{-1/2}\,dr
\leq C\int_{r(0)}^\tau \big(\alpha(r)-\alpha(r(0))\big)^{-1/2}\,dr\\
&\leq C\int_{r(0)}^\tau (r-r(0))^{-1/2}r^{1/2-n}\,dr
\leq Cr(0)^{1-n}\int_1^\infty (u-1)^{-1/2}u^{1/2-n}\,du
\leq Cr(0)^{1-n}.
\end{aligned}
$$
This implies~\eqref{e:stuff-1}.

Next if $p_0\geq 1$ then 
$$
\rho(0)^2\leq 1-\alpha(r(0))^{-2}\leq C r(0)^{2n}
$$
and in this case~\eqref{e:stuff-1} implies~\eqref{e:stuff-2}.

Finally, consider the case $p_0<1$. Since for $0\leq r\leq \tau$,
$1-\alpha(r)^{-2}p_0\geq 1-p_0+{1\over 4}r^{2n}$, we have
$$
t\leq \int_0^{\tau}\frac{dr}{\sqrt{1-\alpha(r)^{-2}p_0}}
\leq \int_0^\infty {dr\over \sqrt{1-p_0+r^{2n}/4}}.
$$
Making the change of variables $r=(4(1-p_0))^{1\over 2n}u$, we get
$$
t\leq C(1-p_0)^{1-n\over 2n}\int_0^\infty {du\over \sqrt{1+u^{2n}}}
\leq C(1-p_0)^{1-n\over 2n}
$$
which implies
$$
1-p_0\leq Ct^{-{2n\over n-1}}.
$$
We now have by~\eqref{e:stuff-1}
$$
\rho(0)^2=1-\alpha(r(0))^{-2}p_0
\leq 1-p_0+Cr(0)^{2n}
\leq Ct^{-{2n\over n-1}}
$$
which gives~\eqref{e:stuff-2}.
\end{proof}
Applying Lemma~\ref{l:stuff}, we obtain the volume bound
$\mu_L(\varphi_{-t}(S^*M\cap\mathcal T_{\mathcal B}(2t)))\leq Ct^{-{n+1\over n-1}}$
and thus
$$
\mathcal V(t)\leq Ct^{-{n+1\over n-1}}.
$$

\bibliography{references.bib}
\bibliographystyle{alphaurl}

\end{document}